\documentclass[10pt, twoside]{amsart}

\usepackage[pagewise]{lineno}

\usepackage{xpatch}
\makeatletter   
\xpatchcmd{\@tocline}
{\hfil\hbox to\@pnumwidth{\@tocpagenum{#7}}\par}
{\ifnum#1<0\hfill\else\dotfill\fi\hbox to\@pnumwidth{\@tocpagenum{#7}}\par}
{}{}
\makeatother    

\makeatletter
 \def\l@subsection{\@tocline{2}{0pt}{4pc}{6pc}{}}
\def\l@subsubsection{\@tocline{3}{0pt}{8pc}{8pc}{}}
 \makeatother

\usepackage[utf8]{inputenc}
\usepackage{hyperref}

\usepackage{graphicx}              
\usepackage{amsmath, bm}               
\usepackage{amsfonts}              
\usepackage{amsthm}                
\usepackage{amssymb}
\usepackage{mathtools}             

\usepackage{enumitem}

\usepackage[all]{xy}
\usepackage{amscd}

\DeclareSymbolFont{symbolsC}{U}{pxsyc}{m}{n}
\DeclareMathSymbol{\coloneqq}{\mathrel}{symbolsC}{"42}
\allowdisplaybreaks
\usepackage{color}
\usepackage{fp}

\newcommand{\Z}{\mathbb{Z}}

\newcommand{\C}{\mathbb{C}}
\newcommand{\R}{\mathbb{R}}

\newcommand{\Li}{\mathcal{B}}
\newcommand{\KG}{\mathcal{K}_G^p}

\newcommand{\cj}{\overline}
\newcommand{\op}{\mathrm}

\DeclarePairedDelimiterX{\normb}[1]{\lVert}{\rVert}{#1}

\numberwithin{equation}{section}
\newtheorem{theorem}{Theorem}[section]
\newtheorem{lemma}[theorem]{Lemma}
\newtheorem{proposition}[theorem]{Proposition}
\newtheorem{corollary}[theorem]{Corollary}

\let\origproofname\proofname
\renewcommand{\proofname}{\upshape\textbf{\origproofname}}

\theoremstyle{definition}

\newtheorem{definition}[theorem]{Definition}
\newtheorem{remark}[theorem]{Remark}
\newtheorem{notation}[theorem]{Notation}

\newtheorem*{theorem*}{Theorem}

\makeatletter
\newtheorem*{rep@theorem}{\rep@title}
\newcommand{\newreptheorem}[2]{%
\newenvironment{rep#1}[1]{%
 \def\rep@title{#2 \ref{##1}}%
 \begin{rep@theorem}}%
 {\end{rep@theorem}}}
\makeatother

\newreptheorem{theorem}{Theorem}

\usepackage{indentfirst}

\begin{document}

\title{Twisted crossed products of Banach algebras}

\author{Alonso Delfín}
\email{alonso.delfin@colorado.edu}
\urladdr{https://math.colorado.edu/\symbol{126}alde9049}

\author{Carla Farsi}
\email{carla.farsi@colorado.edu}
\urladdr{https://www.colorado.edu/math/carla-farsi}

\author{Judith Packer}
\email{judith.jesudason@colorado.edu}
\urladdr{https://math.colorado.edu/\symbol{126}packer}
\address{Department of Mathematics, University of Colorado, Boulder, CO 80309-0395}

\date{\today}

\thanks{\textsc{Department of Mathematics, University of Colorado, Boulder CO 80309-0395, USA}
}

\maketitle

\begin{abstract}
Given a locally compact group $G$, a nondegenerate Banach algebra $A$ with a contractive approximate identity, a twisted action $(\alpha, \sigma)$ of $G$ on $A$, and a family $\mathcal{R}$ of uniformly bounded representations of $A$ on Banach spaces, we define the twisted crossed product $F_\mathcal{R}(G,A,\alpha, \sigma)$.
When $\mathcal{R}$ consists of contractive representations, we show that $F_\mathcal{R}(G,A,\alpha, \sigma)$ is a Banach algebra with a contractive approximate identity, which can also be characterized by an  isometric universal property. 

As an application, we specialize to the $L^p$-operator algebra setting, defining both the $L^p$-twisted crossed product and the reduced version. Finally, we give a generalization of the so--called \textit{Packer–Raeburn trick} to the $L^p$-setting, showing that any $L^p$-twisted crossed product is ``stably'' isometrically isomorphic to an untwisted one.
\end{abstract}
\tableofcontents
\section{Introduction}

Since the late 1950s, the crossed product construction in C*-algebras has been an important tool both to construct novel examples and to study group representations. More generally, works on twisted actions of discrete groups on C*-algebras and on C*-twisted crossed products can be traced back to the next decade in the classical papers \cite{Leptin, Zeller, BusSmi70}. Two main goals of this paper are: (1) to introduce a general definition of crossed products associated with twisted actions of locally compact groups on Banach algebras and (2) to recover certain $L^p$-operator algebra versions of well-known theorems for twisted crossed products of C*-algebras.

Particular cases of twisted crossed products of Banach algebras are already present in \cite{Bardadyn_2024}  for discrete groups acting on locally compact Hausdorff spaces and also in \cite[Section 3]{HetOrt23}, where $F_\lambda^p(G, \sigma)$, the reduced twisted $L^p$-group algebra of a locally compact group $G$, is thoroughly studied and shown to be a full invariant for the pair $(G,\sigma)$ when $p\neq 2$. The general case of a full twisted crossed product of a general Banach algebra by a locally compact group has not yet been considered and involves more technical considerations concerning multiplier algebras. In this paper, we aim to fill this natural gap. To avoid technicalities, we opt to only consider Banach algebras that are nondegenerately represented on a Banach space and that have a contractive approximate identity (cai). Under this setting, for any twisted Banach algebra dynamical system $(G,A, \alpha, \sigma)$ (see Definition \ref{TwistedBanachSystem}) we have defined a universal crossed product $F_{\mathcal{R}}(G,A, \alpha, \sigma)$ associated to each family $\mathcal{R}$ of contractive covariant representations of the system. Specializing to the case when $A$ is an $L^p$-operator algebra and $\mathcal{R}$ consists of all contractive covariant representations on $L^p$-spaces, we give a $p$-analogue of the so-called \emph{Packer–Raeburn trick}, by showing that, when $p \in (1, \infty)$, the full $L^p$-twisted crossed product of $A$ by a locally compact group $G$ is ``stably'' isometrically isomorphic to the full untwisted crossed product of the left spatial tensor product of the compact operators on $L^p(G)$ with $A$. This result requires a deep understanding 
of different tensor products available for $L^p$-operator algebras (see Definitions \ref{defi_LpT_P} and \ref{DefLPOPtp}), which all collapse to the same construction in the C*-setting when one of the factors is the compact operators on $L^2(G)$. We chose to use the left $p$-stabilization of an $L^p$-operator algebra (see Definition \ref{def_stabilization}) for our main results because it is, by construction, independent of the representation of $A$ chosen. 

A brief history of the overall theory of crossed products for Banach algebras begins with the ordinary (i.e., untwisted) crossed products of Banach algebras by locally compact groups associated to different classes of covariant representations, which have been thoroughly studied in the series of (still unpublished\footnote{Prof. Marcel de Jeu kindly informed us that these works were accepted on the condition that they be merged into a single memoir, a task that is still in progress.}) papers \cite{CPBAI_11, CPBAII_13, CPBAIII_13}. A paper specializing only on full and reduced crossed products for Banach algebras with left bounded approximate identities was published shortly after, see \cite{LiXu14}.
At the same time, N.C. Phillips  in \cite{ncp2013CP} adapted  the definitions in \cite{CPBAI_11} to $L^p$-operator algebras, yielding definitions for $F^p(G, A, \alpha)$ and $F^p_{\op{r}}(G, A, \alpha)$, 
the full and reduced crossed products of an $L^p$-operator algebra $A$ by a (second countable) locally compact group $G$. As a particular instance the $L^p$-group algebras $F^p(G)$ and $F^p_{\op{r}}(G)$ are recovered when $G$ trivially acts on $\C$. It is worth mentioning that some C*-results do not always carry to the $L^p$-setting, 
such as Takai duality for $L^p$-crossed products, which is shown to fail in general \cite{WaZh23}.
More recently, Phillips' construction has also been generalized to étale groupoids  \cite{GardLup17} and to both $L^p$-twisted group algebras and reduced twisted groupoid $L^p$-operator algebras of an étale groupoid, see \cite{HetOrt23, BarKwaAnd24, BKM_JFA_25}.
Other papers have appeared that deal with these algebras in various degrees of generality, such as \cite{ChoiGATh24} where the C*-core of the (reduced and untwisted) $L^p$-crossed products by discrete groups is studied, and \cite{Flo25} which deals with the computation of the topological stable rank and real rank for the (reduced) $L^p$-twisted crossed product of a discrete group $G$ of subexponential growth acting on the commutative algebra $C_0(X)$. 

\subsubsection*{Standing assumptions} In this work, we only work with nondegenerate Banach algebras (that is, those that have a two sided cai) and therefore can be nondegenerately represented on a Banach space. This is done in order to make sure that the multiplier algebra of such algebras is well behaved and still an $L^p$-operator algebra when the original Banach algebra is represented on an $L^p$-space. For the most part we let $G$ be any locally compact group, but starting in Section \ref{S:T_LP_CP} we will require $G$ to be second countable, which allows us to use some results from \cite{ncp2013CP}. In our definition of a twisted Banach algebra dynamical system $(G,A, \alpha, \sigma)$ (see Definition \ref{TwistedBanachSystem})
we require the map $\alpha \colon G \to \op{Aut}(A)$ to be strongly continuous and the twist $\sigma \colon G \times G \to \op{Inv}_1(M(A))$ 
to be jointly strictly continuous. Thus, our assumptions on the twisted action are stronger 
than in \cite{PacRae89, PacRae90}, where it is only required that both $\alpha$ and $\sigma$ are Borel maps. This simplifies many arguments, avoids any discussion of Polish groups 
related to Banach algebras, and it allows us to take results from
\cite{CPBAI_11, CPBAII_13, CPBAIII_13, ncp2013CP} in which also the stronger condition of continuity is assumed. 

\subsubsection*{Organization of the paper and main results} In Section \ref{S:Prelim} we establish the  basic notation and present some known facts about multiplier algebras of Banach algebras, as well as basic facts about the spatial tensor product of $L^p$-spaces and tensor products of $L^p$-operator algebras. In Section \ref{S:T_B_C_P}, combining the paths paved out for C*-algebras in \cite{PacRae89} and for untwisted actions on Banach algebras in \cite{CPBAI_11}, we define $(\alpha, \sigma)$, a twisted action of a locally compact group $G$ on a nondegenerate Banach algebra $A$ with a cai, which in turn gives rise to a twisted Banach algebra dynamical system $(G,A,\alpha, \sigma)$. In this section we also construct the $L^1$-convolution algebra $L^1(G,A,\alpha, \sigma)$ associated to a twisted dynamical system, by extending some of the methods used in \cite{BusSmi70, Leptin} for the involutive case.  Mimicking the involutive case, we show in Proposition \ref{BAIL1} that $L^1(G,A,\alpha, \sigma)$ always has a cai. In addition, we 
introduce twisted covariant representations as well as the twisted crossed product $F_{\mathcal{R}}(G, A, \alpha, \sigma)$ associated with a family of uniformly bounded covariant representations $\mathcal{R}$ of $(G,A,\alpha, \sigma)$,  which is obtained as a completion of $L^1(G,A,\alpha, \sigma)$, see Definition \ref{def:R-tw-crossed-pr}.  The main theorem in this section shows that if $C_\mathcal{R}$ is the uniform bound of $\mathcal{R}$, then the twisted crossed product has a $C_\mathcal{R}$-approximate identity: 
\begin{reptheorem}{F_R:BAI}
Let $(G, A, \alpha, \sigma)$ be a twisted Banach algebra dynamical system and let 
$\mathcal{R}$ be a nonempty uniformly bounded class of covariant representations of $(G, A, \alpha, \sigma)$. Then $F_{\mathcal{R}}(G, A, \alpha, \sigma)$ has a $C_\mathcal{R}$-bounded approximate identity. 
\end{reptheorem}

Starting in Section \ref{S:U_P} we restrict to families of contractive covariant representations $\mathcal{R}$ with $C_\mathcal{R}\leq 1$. Here, following the untwisted case \cite[Theorem 2.61]{williams_2007} and \cite[Theorem 4.4]{CPBAII_13}, we obtain, in \textbf{Theorem \ref{Universal_R_Prop}}, a universal property for $F_{\mathcal{R}}(G,A,\alpha, \sigma)$. This universal property, which follows from a careful treatment of both the twist and multiplier algebras,  reduces to the one obtained by the third author and I. Raeburn in \cite[Proposition 2.7]{PacRae89} when $A$ is a C*-algebra and $\mathcal{R}$ consists of all the covariant representations of $(G,A,\alpha, \sigma)$ acting on Hilbert spaces. It also reduces to the universal property obtained for untwisted Banach algebra dynamical systems in \cite[Theorem 4.4]{CPBAII_13}. The main results in our final section, Theorem \ref{Ext_Equiv_ISOM} and Theorem \ref{PacRae-TRICK}, are both consequences of our general universal property. 

In Section \ref{S:T_LP_CP} we specialize to the cases where the contractive covariant representations acting on $L^p$-spaces give rise to $F^p(G,A, \alpha, \sigma)$, the full twisted $L^p$-crossed product, and also 
to $F_{\op{r}}^p(G,A, \alpha, \sigma)$, the reduced twisted $L^p$-crossed product. A natural question, which we do not address in this paper, is to determine if the amenability results (such as \cite[Theorem 3.7]{GArThi15}, \cite[Lemma 7.5]{ChoiGATh24}, and \cite[Theorem 5.5]{HetOrt23}) can be generalized to the twisted setting. In Section \ref{S:P_R_T} we introduce a general notion of exterior equivalence for twisted actions. Roughly speaking, 
two twisted actions $(\alpha, \sigma)$ and $(\beta, \omega)$ of $G$ on $A$ are exterior equivalent via $\theta  \colon G \to \op{Inv}_1(M(A))$ if 
\begin{enumerate}
\item $x\mapsto \theta_xa$ and $x\mapsto \theta_xa$ are continuous maps $G \to A$ for all $a \in A$,
\item $\beta_x = \op{Ad}(\theta_x)\circ \alpha_x$ for all $x \in G$, 
\item $\omega_{x,y}\theta_{xy}=\theta_x\alpha_x(\theta_y)\sigma_{x,y}$.
\end{enumerate}
Given an exterior equivalence, to each family $\mathcal{R}$ of contractive representations of $(G,A, \alpha, \sigma)$ 
we associate a family $\mathcal{R}_\theta$ of contractive representations of $(G,A, \beta, \omega)$. 
We then prove one of our main results about how equivalent twisted actions induce isometrically isomorphic twisted crossed products: 
\begin{reptheorem}{Ext_Equiv_ISOM}
Let $G$ be a locally compact group, let 
$A$ be a nondegenerate Banach algebra with a cai, and let $(\alpha, \sigma)$ and $(\beta, \omega)$ be two exterior equivalent twisted actions of $G$ on $A$ via $\theta$. Then the crossed product $F_{\mathcal{R}}(G, A, \alpha, \sigma)$ is isometrically isomorphic 
to $F_{\mathcal{R_\theta}}(G, A, \beta, \omega)$.
\end{reptheorem}
As a consequence, in Corollary \ref{Cor:LpExtEquiv} we show that exterior equivalent twisted actions yield
 isometrically isomorphic full twisted $L^p$-crossed products, which is then used to 
produce our final main theorem concerned with a general $p$-version of the Packer-Raeburn trick. 
For the $p$-version of the untwisting trick, we work with a left $p$-stabilization of the algebra $A$ by the compact operators $\KG \coloneqq \mathcal{K}(L^p(G))$. That is, 
we consider the left spatial tensor product $\KG \otimes_{\op{lsp}}^p A$, which is constructed by
 fixing the canonical representation of $\KG \subseteq \mathcal{B}(L^p(G))$ and defined so that the norm on  $\KG \otimes_{\op{lsp}}^p A$ is independent of the representation 
chosen for the $L^p$-operator algebra $A$. 
\begin{reptheorem}{PacRae-TRICK}
Let $p \in (1,\infty)$, let $A$ be a nondegenerate separably representable $L^p$-operator algebra with a cai, 
and let $(G, A, \alpha, \sigma)$ be a twisted dynamical system. 
Then there is an action $\beta$ of $G$ on $\KG \otimes_{\op{lsp}}^p A$
such that $\KG \otimes_{\op{lsp}}^p F^p(G, A, \alpha, \sigma)$ 
is isometrically isomorphic to the untwisted crossed product 
$F^p(G, \KG \otimes_{\op{lsp}}^p A , \beta)$. 
\end{reptheorem}
An immediate consequence of the untwisting trick, that we state in Corollary \ref{cor: Ktheory}, is that the 
$K$-theory of twisted $L^p$-crossed products can be computed with the already developed 
techniques for untwisted $L^p$-crossed products. 

Finally, we also mention that Theorem \ref{Ext_Equiv_ISOM} also holds for the reduced twisted crossed product $F_{\op{r}}^p(G,A, \alpha, \sigma)$. An interesting question, which we do not tackle in this project, is whether the converse holds when $p \neq 2$. That is, does an isometric isomorphism 
between $F_{\op{r}}^p(G,A, \alpha, \sigma)$ and $F_{\op{r}}^p(G,A, \beta, \omega)$ imply that $(\alpha, \sigma)$ and $(\beta, \omega)$ are exterior equivalent? A positive answer to this problem will add to the recent list of rigidity results for $p \neq 2$ such as  \cite[Proposition 5.5]{GArTHi22}, \cite[Theorem 6.7]{ChoiGATh24}, and \cite[Theorem 4.9]{HetOrt23}.

\section*{Acknowledgments} This research was partly supported by: Simons Foundation Collaboration Grant MPS-TSM-00007731 (C.F.) and by the Simons Foundation Collaboration Grant \#1561094 (J.P.). The first author (A.D.) would like to thank Felipe I. Flores for helpful conversations on Banach bundles. The authors would like to thank John Quigg for helpful comments on an earlier version of this manuscript. They would also like to thank the anonymous referee for insightful suggestions that led to improvements of the main results in Section \ref{S:P_R_T}.

\section{Preliminaries}\label{S:Prelim}

\begin{notation}
Let $E$ and $F$ be Banach spaces. We write $\Li(E,F)$ for the space of 
bounded linear maps $a \colon E \to F$, which is a Banach space when equipped with the usual 
operator norm
\[
\| a \|\coloneqq \sup \{ \| a\xi\|_F \colon \| \xi\|_E=1\}.
\]
We denote $\op{Iso}(E) \subseteq \Li(E)\coloneqq \Li(E,E)$ to the group of invertible 
isometries on $E$. The algebra of compact operators on $E$ will be denoted by $\mathcal{K}(E)$. 
\end{notation}

\begin{notation}
We say that a Banach algebra $A$ has a \textit{bounded approximate identity} when there is a net $(e_\lambda)_{\lambda \in \Lambda}$ in $A$ and a constant $C \in \R_{>0}$ such that 
$\| e_\lambda\| \leq C$ 
for all $\lambda \in \Lambda$ and 
such that for each $a \in A$,
\[
\lim_{\Lambda} \| e_\lambda a - a \| = \lim_{\Lambda} \| ae_\lambda  - a \| =0.
\]
 In such case we say that $A$ is \textit{$C$-approximately unital} or that $A$ has a \textit{$C$-approximate identity}. When $C=1$, we say that $A$ has a \textit{contractive approximate identity}, henceforth denoted simply as a cai. 
\end{notation}
\begin{notation}\label{notaRep}
A \textit{representation of a Banach algebra $A$ on a Banach space $E$} will always mean 
a bounded algebra homomorphism $\pi \colon A \to \Li(E)$. We say 
\begin{enumerate}
\item $\pi$ is \textit{contractive} when $\| \pi \| \leq 1$, 
 \item $\pi$ is \textit{isometric} when $\| \pi (a)\|=\|a\|$ for all $a \in A$, 
  \item \label{sep_pi} $\pi$ is \textit{separable} when $E$ is a separable Banach space,
 \item \label{nondeg_pi} $\pi$ is \textit{nondegenerate} when $\pi(A)E \coloneqq \op{span}\{ \pi(a)\xi \colon a \in A, \xi \in E \}$ is a dense subspace of $E$,
 \item $\pi$ is \textit{unital} when $A$ is unital and $\pi(1_A)=\op{id}_E$.
\end{enumerate}
We say $A$ is \textit{representable} on the Banach space $E$ when there is an isometric representation of $A$ on $E$. 
\end{notation}
To achieve the greatest degree of generality,
unless otherwise noted, a Banach algebra $A$ 
will satisfy the next assumption:
\begin{enumerate}[label=(A.\arabic*)]
\item $A$ has a cai.\label{AhasCAI}
\end{enumerate}
An immediate consequence of \ref{AhasCAI} is that $A$ acts nondegenerately on itself via left multiplication (see \cite[Lemma 2.4]{BliDelWel24}). It is therefore safe to also assume that our Banach algebras will always satisfy the following condition:
\begin{enumerate}[label=(A.\arabic*)]
  \setcounter{enumi}{1}
\item $A$ is \textit{nondegenerately representable}. That is, there is a Banach space $E$ and an isometric nondegenerate representation of $A$ on $E$. We often will simply say that $A$ is \textit{nondegenerate}. \label{AhasND_REP}
\end{enumerate}

Recall that Assumption \ref{AhasCAI} is enough to isometrically identify $A$ with a closed
two-sided ideal of $M(A)$, the multiplier algebra of $A$ defined as the algebra of double centralizers for $A$ (see \cite[Section 3]{BliDelWel24} for the general definition and properties of $M(A)$). Indeed, 
this is done via the map $\iota_A \colon A \to M(A)$ defined as 
\begin{equation}\label{iota_A_map}
\iota_A(a) \coloneqq (L_a, R_a),
\end{equation}
where $L_a(b) \coloneqq ab$ and $R_a(b) \coloneqq ba$ for all $a,b \in A$ (see \cite[ Lemma 3.3]{BliDelWel24} for complete details).

It is also well known that Assumption \ref{AhasCAI} above implies that if $\pi$ 
is a nondegenerate representation of $A$ on $E$ (see Notation \ref{notaRep} \eqref{nondeg_pi}), then it extends uniquely to a
unital nondegenerate representation of $M(A)$ on $A$, 
\begin{equation}\label{Eq:hat_pi}
\widehat{\pi} \colon M(A) \to \Li(E),
\end{equation}
satisfying $\widehat{\pi} \circ \iota_A = \pi$ (see \cite[Theorem 4.1, Remark 4.2]{GarThi20} and \cite[Theorem 6.1]{CPBAI_11}). Furthermore, $\widehat{\pi}$ is contractive when $\pi$ is and $\widehat{\pi}$ is isometric when $\pi$ is. This guarantees, using Assumption \ref{AhasND_REP}, an isometric isomorphism
\begin{equation}\label{Ma=Dc}
M(A) \cong \{ t \in \Li(E) \colon t\pi(a) \in \pi(A), \pi(a)t \in \pi(A)\},
\end{equation}
as shown in \cite[Corollary 3.6]{BliDelWel24}. Therefore, after identifying $A$ with its isometric copy 
in $\Li(E)$, each $t \in M(A)$ naturally gives two maps $A \to A$, each one with norm $\|t\|$, defined by $a \mapsto ta \in A$ and $a \mapsto at \in A$. Thus, in this case $M(A)$ is isometrically identified with a subalgebra of $\Li(A)$ via the map $ t \mapsto (a \mapsto ta )$. Furthermore, if $u \in M(A)$ is invertible, we get a map $\op{Ad}(u) \colon A \to A$ 
defined by 
\[
\op{Ad}(u)a \coloneqq uau^{-1}.
\] 
Moreover, since $A$ is canonically embedded as a closed two-sided ideal in $M(A)$, 
any bounded automorphism $\varphi \colon A \to A$ lifts to 
a map $\varphi \colon M(A) \to M(A)$ defined, for any $t \in M(A)$, $a \in A$, 
via
\[
\varphi(t)\varphi(a) = \varphi(t a), \ \varphi(a)\varphi(t) = \varphi(a t).
\]
\begin{notation}
Let $p \in [1,\infty)$ and let $(\Omega, \mathfrak{M}, \mu)$ be a measure space. 
As usual $L^p(\Omega, \mathfrak{M}, \mu)$ is used to denote the space
of $p$-integrable complex valued measurable functions  modulo functions that
vanish a.e.\ $[\mu]$. We will often simply write either $L^p(\mu)$ or $L^p(\Omega,\mu)$
instead of $L^p(\Omega, \mathfrak{M}, \mu)$. 
\end{notation}

\begin{definition}\label{Lp}
Let $p \in [1, \infty)$. A Banach algebra $A$ is an 
\textit{$L^p$-operator algebra} if there is a measure space
$(\Omega, \mathfrak{M}, \mu)$ such that $A$ is represented
on $L^p(\mu)$.  That is, there is an isometric representation $\pi \colon A \to \Li(L^p(\mu))$. 
If $\pi$ is nondegenerate, we say $A$ is a \textit{nondegenerate $L^p$-operator algebra}.
\end{definition}

We refer the reader to \cite[Section 2.3]{Ryan02}
for definitions and facts about Banach space valued measurable 
functions and Bochner integrals. 
We will work extensively with the $p$-Bochner integrable functions: 
\begin{definition}\label{Lebegue-Bochner}
Let $p \in [1, \infty)$, let $(\Omega, \mathfrak{M}, \mu)$ be a measure space, and let $E$ be a Banach space. The \textit{Lebesgue-Bochner $L^p$-space of $(\Omega, \mu; E)$} is 
\[
L^p(\mu; E) \coloneqq L^p(\Omega, \mu; E) \coloneqq 
\{ \xi \colon \Omega \to E \mbox{ measurable} 
\colon 
\textstyle{\int_{\Omega}} \|\xi(\omega)\|^p \  d\mu(\omega) < \infty\}.
\]
Of course $L^p(\mu)=L^p(\mu; \C)$. 
\end{definition}

For $p \in [1, \infty)$, there is a tensor product, called the 
\textit{spatial tensor product} and denoted by $\otimes_p$. 
We refer the reader to \cite[Section 7]{defflor1993} for complete 
details on this 
tensor product. We only describe below the properties we will need. 
If $(\Omega_0, \mathfrak{M}_0, \mu_0)$ is a measure space and $E$ is a 
Banach space, 
then there is an isometric isomorphism 
\[
L^p(\mu_0) \otimes_p E \cong L^p(\mu_0; E), 
\]
such that for any $\xi \in L^p(\mu_0)$ and $\eta \in E$, 
the elementary tensor $\xi \otimes \eta $ 
is sent to the function $\omega \mapsto \xi(\omega)\eta$. 
Furthermore, if $(\Omega_1,  \mathfrak{M}_1, \mu_1)$
is another measure space and $E=L^p(\mu_1)$, then there is 
an isometric isomorphism
\begin{equation}\label{LpTensor}
L^p(\mu_0) \otimes_p L^p(\mu_1) \cong 
L^{p}(\Omega_0 \times \Omega_1, \mu_0 \times \mu_1),
\end{equation}
sending $\xi \otimes \eta$ to the function 
$(\omega_0,\omega_1) \mapsto \xi(\omega_0)\eta(\omega_1)$ 
for every $\xi \in L^p(\mu_0)$
and $\eta \in L^p(\mu_1)$. 
We describe its main properties below. 
The following is part of Theorem 2.16 in \cite{ncp2012AC}, 
except that we have removed the $\sigma$-finiteness assumption 
as in the proof in Theorem 1.1 in \cite{figiel1984}. 
\begin{theorem}\label{SpatialTP}
Let $p \in [1, \infty)$ and for $j \in \{0,1\}$ let 
$(\Omega_j, \mathfrak{M}_j, \mu_j)$, 
$(\Lambda_j, \mathfrak{N}_j, \nu_j)$ be measure spaces.
\begin{enumerate}
\item Under the identification in Equation \eqref{LpTensor}, the algebraic tensor product $L^p(\mu_0) \otimes L^p(\mu_1) = \op{span}\{ \xi \otimes \eta  \colon \xi \in L^p(\mu_0), \eta \in L^p(\mu_1)\}$ 
is a dense subset of 
$L^{p}(\Omega_0 \times \Omega_1, \mu_0 \times \mu_1)$.  \label{denseTensor_p}
\item $\| \xi \otimes \eta \|_p = \| \xi\|_p\|\eta\|_p$  for every 
$\xi \in L^p(\mu_0)$ and $\eta \in L^p(\mu_1)$. \label{LpTensorNorm}
\item Let
$a \in \Li(L^p(\mu_0), L^p(\nu_0))$
and let $b \in \Li(L^p(\mu_1), L^p(\nu_1))$. 
Then there is a unique map 
$a\otimes b \in \Li(L^p(\mu_0 \times \mu_1),
 L^p( \nu_0 \times \nu_1))$ 
such that $\| a\otimes b\|=\|a\|\|b\|$ and such that for every $\xi \in L^p(\mu_0)$ and $\eta \in L^p(\mu_1)$,\label{LpTensorOp}
\[
(a \otimes b)(\xi \otimes \eta)=a\xi \otimes b\eta.
\] 
\item The concrete tensor product of operators defined in 
\eqref{LpTensorOp} is associative, bilinear, 
and satisfies (when the domains are appropriate) 
$(a_1 \otimes b_1)(a_2 \otimes b_2) = a_1 a_2 \otimes b_1b_2$. \label{LpTensorOpProp}
\end{enumerate}
\end{theorem}
Part \eqref{LpTensorOp} in the previous Theorem gives a 
concrete tensor product 
between any two $L^p$-operator algebras:
\begin{definition}\label{defi_LpT_P}
Let $p \in [1, \infty)$ and let 
 $A \subseteq \Li(L^p(\mu_0))$ and $B \subseteq \Li(L^p(\mu_1))$ be $L^p$-operator algebras. 
 We define $
  A \otimes_p B$
   to be the closed linear span, in $\mathcal{B}\big(L^p(\mu_0 \times \mu_1) \big)$, of all $a \otimes b$ for $a \in A$ and $b \in B$. 
\end{definition}

Definition \ref{defi_LpT_P} provides only a concrete tensor product of $L^p$-operator algebras. 
Different representations for $A$ and $B$ on $L^p$-spaces can yield a different tensor product, 
as it was shown below Example 1.15 in \cite{ncp2013CP}. This issue appears even when $p=2$, in the nonselfadjoint case, but does not happen for C*-algebras. To fix this, in 
\cite[Definition 7.2]{ChoiGATh24}, Choi-Gardella-Thiel introduced a general theory of tensor products for $L^p$-operator algebras. Indeed, for $p \in [1, \infty)$, denote by $\op{Rep}_p(A)$ to the class of all the contractive nondegenerate representations of $A$ on $L^p$-spaces.

\begin{definition}\label{DefLPOPtp}
Let $p \in [1, \infty)$ and let  $A$ and $B$ be $L^p$-operator algebras. The \textit{$L^p$-operator spatial tensor product} $A \otimes_{\op{sp}}^p B$ is defined as the completion of $A \otimes B$
under the norm 
\[
A \otimes B \ni t \mapsto \| t \|_{\op{sp}} \coloneqq \sup\{ \| (\pi_A \otimes \pi_B)(t)\| \colon \pi_A \in \op{Rep}_p(A), \pi_B \in \op{Rep}_p(B)\}.
\] 
\end{definition}

\begin{remark}
 Notice that if $A \subseteq \Li(L^p(\mu_0))$ and $B \subseteq \Li(L^p(\mu_1))$ are nondegenerate concrete $L^p$-operator algebras (i.e. $\cj{AL^p(\mu_0)}=L^p(\mu_0)$ and $\cj{BL^p(\mu_1)}=L^p(\mu_1)$), then by construction the identity map on the algebraic tensor product $A \otimes B$ extends to a contraction $A \otimes_{\op{sp}}^p B \to A \otimes_p B$ with dense range.
 \end{remark}
 
 In this paper, we only work with tensor products of $L^p$-operator algebras where one of the factors is 
 the compact operators acting on an $L^p$-space. The setting is as follows, let $p \in [1, \infty)$, let $(\Omega, \mathfrak{M}, \mu)$ be a measure space, and let $A$ be a separable $L^p$-operator algebra with a cai. 
 The algebra $\mathcal{K}(L^p(\mu))$ has a canonical isometric nondegenerate representation on $L^p(\mu)$ 
 coming from the inclusion $\iota_{\mathcal{K}}  \colon \mathcal{K}(L^p(\mu)) \hookrightarrow \mathcal{B}(L^p(\mu))$. We fix this representation 
 and put a norm on the algebraic tensor product $\mathcal{K}(L^p(\mu)) \otimes A$ as follows:
 \[
\mathcal{K}(L^p(\mu))  \otimes A \ni t \mapsto \| t \|_{\op{lsp}} \coloneqq \sup\{ \| (\iota_{\mathcal{K}} \otimes \pi_B)(t)\| \colon \pi_A \in \op{Rep}_p(A))\}.
\] 
We call this norm the \textit{$L^p$-operator left spatial tensor product} and the completion of $\mathcal{K}(L^p(\mu)) \otimes A$ under this norm will be denoted as $\mathcal{K}(L^p(\mu))  \otimes_{\op{lsp}}^p A$. Once again, by construction we get contractions with dense range 
\[
\mathcal{K}(L^p(\mu))\otimes_{\op{sp}}^p A \to \mathcal{K}(L^p(\mu)) \otimes_{\op{lsp}}^p A \to \mathcal{K}(L^p(\mu)) \otimes_p A,
\]
that is $\| t \| \leq  \| t \|_{\op{lsp}} \leq  \| t \|_{\op{sp}}$ for any $t \in \mathcal{K}(L^p(\mu)) \otimes A$. 
\begin{definition}\label{def_stabilization}
Let $p \in [1, \infty)$, let $(\Omega, \mathfrak{M}, \mu)$ be a measure space and let $A$ be a separable $L^p$-operator algebra with a cai. We define the \textit{spatial $p$-stabilization of $A$} by 
\[
\op{St}^p(A) \coloneq \mathcal{K}(L^p(\mu))  \otimes_{\op{sp}}^p A,
\]
and the \textit{left spatial $p$-stabilization of $A$} by 
\[
 \op{St}_{\op{l}}^p(A) = \op{St}_{\op{l}, \Omega}^p(A)\coloneq \mathcal{K}(L^p(\mu))  \otimes_{\op{lsp}}^p A.
\]
\end{definition}

\begin{remark}\label{rmk_stab}
By construction, $\op{St}_{\op{l}}^p(A)$ is independent of the representation of $A$ chosen. However, it might still depend on the representation of $\mathcal{K}(L^p(\mu))$. If $p=2$ and $A$ is a C*-algebra, then $\| t \|_{\op{sp}} = \| t \|_{\op{lsp}}$ and therefore $\op{St}_{\op{l}}^2(A)$ isometrically coincides with $\op{St}^p(A)$ making it also independent of the representation chosen for $\mathcal{K}(L^2(\mu))$. We do not know whether this 
might still hold for general $L^p$-operator algebras. In order to prove a result along these lines, it would be sufficient to show that any representation $\pi_{\mathcal{K}} \colon \mathcal{K}(L^p(\Omega,\mu)) \to \mathcal{B}(L^p(\Omega_0, \mu_0))$ is, in some suitable sense, equivalent to one of the form $u \mapsto u \otimes \op{id}_{L^p(\Omega_1,\mu_1)}$ acting on $L^p(\Omega, \mu) \otimes_p L^p(\Omega_1, \mu_1)$ for some measure space $(\Omega_1, \mu_1)$. This is not clear to be true in general. On the other side, for the particular case of $L^p(\Omega,\mu)=\ell^p(I)$ for $I$ countable and $A$ having unique $L^p$-operator matrix norms, it follows from \cite[Lemma 2.4]{WaZh23} that $\op{St}_{\op{l}, I}^p(A)=\mathcal{K}(\ell^p(I))  \otimes_{p} A$, making the concrete tensor product from Definition \ref{defi_LpT_P}  independent of the representation for $A$ chosen. 
\end{remark}
In order to work with the minimal assumptions needed to get a stabilization that is independent of the representation of $A$, in Section \ref{S:P_R_T}
we will only work with $\op{St}_{\op{l}}^p(A)$, the left spatial $p$-stabilization of $A$. As discussed in Remark \ref{rmk_stab} above, all of our results will remain true if we work with a discrete countable group $G$,  an algebra $A$ with unique $L^p$-operator matrix norms, and use the concrete spatial tensor product $\mathcal{K}(\ell^p(G)) \otimes_p A$. More generally, our results will automatically hold for the spatial $p$-stabilization $\op{St}_{\op{sp}}^p(A)$ if we can show the left  stabilization is independent of the representation chosen for $\mathcal{K}(L^p(\mu))$. 

\section{Twisted Banach Crossed Products}\label{S:T_B_C_P}
 
Very much in the sense of \cite{BusSmi70}, we now define a twisted Banach $L^1$-algebra associated to a twisted action of a locally compact group $G$ on $A$. To do so, we let $\op{Aut}(A)$
be the group of isometric automorphisms of $A$ (i.e., $\varphi \in \op{Aut}(A)$ if $\varphi \colon A \to A$   is an isometric invertible algebra homomorphism) and $\op{Inv}_1(M(A))$ be the group of invertible elements  $ u \in M(A)$ with $\| u \| = \| u^{-1}\| = 1$.

\begin{definition}\label{TwistedBanachSystem}
Let $G$ be a locally compact group. A \textit{twisted action} of $G$ on $A$ consists of a 
pair $(\alpha, \sigma)$ where 
\begin{align*}
 \alpha  \colon G &\to \op{Aut}(A) & \sigma \colon G \times G & \to \op{Inv}_1(M(A))\\
  x & \mapsto \alpha_x \coloneqq\alpha(x)  & (x,y) & \mapsto \sigma_{x,y}\coloneqq \sigma(x,y)
\end{align*}
 satisfy
\begin{enumerate}
\item\label{strongly_alpha} for each $a \in A$, $x \mapsto \alpha_x(a)$ is a continuous map $G \to A$,
\item\label{strict}  for each $a \in A$, both $(x,y) \mapsto \sigma_{x,y}a$ and 
 $(x,y) \mapsto a\sigma_{x,y}$ are continuous maps $G \times G \to A$,
\item $\alpha_{1_G} = \op{id}_A$, 
\item  for all $x \in G$, $\sigma(1_G, x)=\sigma(x,1_G)=\op{id}_{M(A)}$,\label{sigma_1=1}
\item for all $x,y \in G$, $\alpha_x\alpha_y = \op{Ad}(\sigma_{x,y})\circ \alpha_{xy}$,\label{x,yAd}
\item  for all $x,y,z \in G$, $\alpha_{z}(\sigma_{x,y})\sigma_{z,xy} = \sigma_{z,x}\sigma_{zx,y}$.\label{ASzxy}
\end{enumerate} If $(\alpha, \sigma)$ is a twisted action of $G$ on $A$, we call $(G, A, \alpha, \sigma)$ a \textit{twisted Banach algebra dynamical system}.
\end{definition}
 \begin{remark} We point out that in several computations below it will be used implicitly that $\|\alpha_x(a)\|=\|a\|$ for all $x \in G$ and $a \in A$. 
\end{remark} 
\begin{remark}\label{Rmk_Strong_vs_borel}
Condition \eqref{strongly_alpha} means that $\alpha$ is strongly continuous. When equipping $M(A)$ with its strict topology, Condition \eqref{strict} above means that $\sigma$ is 
jointly strictly continuous. We warn the reader that in \cite{PacRae89, PacRae90} it is only required that both 
$\alpha$ and $\sigma$ are Borel maps. In our case, we require the stronger assumption of continuity to be able to draw some results from \cite{CPBAI_11,ncp2013CP}, and also to avoid any discussion of Polish groups associated with Banach algebras. 
\end{remark}

Let $\nu_G$ be a fixed left Haar measure for $G$ and as usual we write $dx \coloneqq  d\nu_G(x)$ for $x \in G$. Let $\Delta \colon G \to (0,\infty)$ be the modular function for $G$, that is 
$\Delta$ is a continuous group homomorphism 
satisfying, for any $\psi \in L^1(G)\coloneqq L^1(G,\nu_G)$, 
\[
\int_G \psi(xy) dx = \Delta(y^{-1}) \int_G \psi(x)dx.
\]
 Let $L^1(G; A) \coloneqq L^1(G, \nu_G; A)$ be as in Definition \ref{Lebegue-Bochner}, so that  $L^1(G; A)$ consists of all the measurable functions $f \colon G \to A$ such that 
\[
\| f \|_1 \coloneqq \int_{G} \| f(x) \| dx<\infty.
\]
Of course, $L^1(G;\C)=L^1(G)$. Given a twisted action $(\alpha, \sigma)$ of $G$ on $A$, we now let $L^1(G, A, {\alpha, \sigma})$ be $L^1(G;A)$ equipped
 with the twisted convolution $*_{\alpha,\sigma}$ as multiplication, that is
\[
(f*_{\alpha, \sigma} g)(x) \coloneqq \int_G f(y) \alpha_{y}(g(y^{-1}x))\sigma_{y,y^{-1}x} dy.
\]
\begin{remark}
If $G$ is discrete and $A=C_0(X)$ for a locally compact Hausdorff space $X$, then 
$L^1(G, A, {\alpha, \sigma})$ coincides with the algebra $F(\alpha, \sigma)$ from \cite[Definition 2.1]{Bardadyn_2024}.
\end{remark}
As usual we denote by $C_c(G,A)$ the set of functions 
$f \colon G \to A$ with compact support.  If $C_c(G,A)$ is equipped with the twisted convolution $*_{\alpha, \sigma}$, we denote the resulting algebra by $C_c(G,A,\alpha, \sigma)$, which is clearly a
dense subalgebra of $L^1(G,A,\alpha, \sigma)$. It will be convenient to also equip $C_c(G,A,\alpha, \sigma)$ with the inductive limit topology, which we precisely define below.

\begin{definition}\label{InductiveLtop}
Let $G$ be a locally compact group, let $A$ be a Banach algebra, and let $\mathcal{K}=\{K \subseteq G \colon K \text{ is compact }\}$. For each $K \in \mathcal{K}$, equip $C_K(G,A) \coloneqq \{ f \in C_c(G,A) \colon \cj{\op{supp}}(f) \subseteq K\}$ with the norm $\| f \|_K \coloneqq \sup\{\| f(x)\| \colon {x \in K}\}$. The  \textit{inductive limit topology} on $C_c(G,A)$ is the largest topology making the inclusion maps $C_K(G,A)  \hookrightarrow C_c(G,A)$ continuous for all $K \in \mathcal{K}$.
\end{definition}

When $A$ is a C*-algebra, is well known that $L^1(G,A,\alpha, \sigma)$ has a cai (see discussion before Proposition 2.7 in \cite{PacRae89} or the Appendix in \cite{PacRae90}).
This is still true for a general Banach algebra $A$:

\begin{proposition}\label{BAIL1}
Let $(G, A, \alpha, \sigma)$ be a twisted Banach algebra dynamical system.
Then $C_c(G,A,\alpha, \sigma)$ has a cai with respect to the inductive limit topology. 
Moreover, such cai is also a cai for $L^1(G,A,\alpha, \sigma)$.
\end{proposition}
\begin{proof}
Let $G \ltimes_{\alpha, \sigma} A$ be defined as $G \times A$ equipped with the product topology together with a multiplication $\cdot_{\alpha, \sigma}$ given by 
\[
(x,a) \cdot_{\alpha, \sigma} (y,b) \coloneqq (xy,a\alpha_x(b)\sigma_{x,y}),
\]
for all $(x,a), (y,b) \in G \times A$. Then  $\mathcal{A} \coloneqq(G \ltimes_{\alpha, \sigma} A, \pi )$ is a Banach bundle over $G $ where  $\pi \colon G \ltimes_{\alpha, \sigma} A \to G$  
is given by $\pi(x,a)\coloneqq x$ and each fiber is equipped with the constant $A$-norm 
\[
\|(x,a)\|_{\pi^{-1}(x)} \coloneqq \|a\|.
\]
Let $(e_\lambda)_{\lambda \in \Lambda}$ be a cai for $A$. We claim that $\big( (1_G, e_\lambda) \big)_{\lambda \in \Lambda}$ is a strong cai for $\mathcal{A}$ in the sense of \cite[VIII-Definition 2.11]{FellDor126}. 
That is, $\| (1_G,e_\lambda) \|_{\pi^{-1}(1_G)}\leq 1$ and for every compact subset $Q \subseteq G \ltimes_{\alpha, \sigma} A$ and any 
$\varepsilon>0$ there is $\lambda_0 \in \Lambda$ such that 
\[
\| (x,a)\cdot_{\alpha, \sigma}(1_G,e_\lambda)-(x,a)\|_{\pi^{-1}(x)}<\varepsilon,
\]
and 
\[
\| (1_G,e_\lambda)\cdot_{\alpha, \sigma}(x,a)-(x,a)\|_{\pi^{-1}(x)}<\varepsilon 
\]
for all $(x,a) \in Q$ and all $\lambda \geq \lambda_0$. 
Observe that 
\[
(x,a)\cdot_{\alpha, \sigma}(1_G,e_\lambda)-(x,a)= (x,e_\lambda a - a),
\]
and 
\[
 (1_G,e_\lambda)\cdot_{\alpha, \sigma}(x,a)-(x,a)= (x,a \alpha_x(e_\lambda)-a).
\]
Hence, the claim follows at once from the fact that $(e_\lambda)_{\lambda \in \Lambda}$ is a (strong) cai for $A$ and two standard compactness arguments. In fact, the details are exactly as the ones one will need to carry in the untwisted case (see \cite[VIII-4.3]{FellDor126} ), for the two displayed 
multiplications above are independent of the twist chosen.

Next, let $L^1(G \mid \mathcal{A})$ be the cross-sectional algebra of $\mathcal{A}$ as defined in Chapter VIII Section 5 of \cite{FellDor126}. 
Since $\mathcal{A}$ has a strong cai, we apply \cite[VIII-Theorem 5.11]{FellDor126} and get 
a (strong) cai for $L^1(G \mid \mathcal{A})$ that belongs to 
$C_c(G \mid \mathcal{A})$, the space of continuous sections $G \to G \ltimes_{\alpha, \sigma} A$ with compact support. 
Furthermore, by \cite[VIII-Remark 5.12]{FellDor126}, the net obtained is also a cai for $C_c(G \mid \mathcal{A})$ when equipped with the 
 inductive limit topology (see \cite[II-14.3]{FellDor125}).
 Finally, observe that the map  $L^1(G, A, \alpha, \sigma) \to L^1(G \mid \mathcal{A})$, 
 $f \mapsto \widetilde{f}$, where $\widetilde{f} \colon G \to G \ltimes_{\alpha, \sigma} A$ is given by 
 \[
 \widetilde{f}(x) \coloneqq (x,f(x)),
 \]
  is an isometric algebra isomorphism that identifies $C_c(G,A,\alpha, \sigma)$ with $C_c(G \mid \mathcal{A})$, finishing the proof. 
 \end{proof}
 
\begin{remark}\label{ExplicitCAI}
For convenience and future use, we explicitly define the cai of $L^1(G \mid \mathcal{A})$ obtained in the previous proof in the setting of $L^1(G,A,\alpha, \sigma)$. Let $(e_\lambda)_{\lambda \in \Lambda}$ 
 be the cai for $A$ and let $\mathcal{U} \subseteq \mathcal{P}(G)$ be a neighborhood base of $1_G$. For each $U \in \mathcal{U}$ fix a function $\psi_U \in C_c(G)$ such that $\cj{\op{supp}}( \psi_U )\subseteq U$, 
 $\psi_U \geq 0$, $\psi_U(x^{-1})=\psi_U(x)$ for all $x \in G$, and $\int _G \psi_U(x)dx=1$. 
 For each $\lambda \in \Lambda$, $U \in \mathcal{U}$ define 
  $f_{\lambda, U} \colon G \to A$ by 
 \[
 f_{\lambda, U}(x) \coloneqq \psi_U(x) e_\lambda.
 \]
 Then $(f_{\lambda, U})_{\lambda \in \Lambda, U \in \mathcal{U}}$ 
 is the desired cai in $C_c(G,A,\alpha, \sigma)$. 
 \end{remark}
For a Banach space $E$, recall that $\op{Iso}(E) \subseteq \Li(E)$ denotes the group of invertible 
isometries on $E$. We now turn to covariant representations of twisted Banach algebra dynamical systems. 

\begin{definition}\label{CovariantRep}
A \textit{covariant representation} of a twisted Banach algebra dynamical system $(G, A, \alpha, \sigma)$
on a Banach space $E$ consists of a pair $(\pi, u)$ where 
\begin{enumerate}
\item $\pi \colon A \to \Li(E)$ is a nondegenerate representation of $A$ on $E$, \label{CovC1}
\item $u \colon G \to \op{Iso}(E)$ is a strongly continuous map, \label{CovC2}
\item $u_xu_y = \widehat{\pi}(\sigma_{x,y})u_{xy}$ for all $x,y \in G$, (see Equation \eqref{Eq:hat_pi})\label{CovC3}
\item $\pi(\alpha_x(a))=u_x\pi(a)u_x^{-1}$ for all $x \in G$, $a \in A$. \label{CovC4}
\end{enumerate}
\end{definition}

Once again as in Remark \ref{Rmk_Strong_vs_borel}, and to contrast with \cite[Definition 2.3]{PacRae89}, we opt to follow \cite{CPBAI_11,ncp2013CP} and impose the stronger assumptions 
of continuity for $u$ rather than simply asking for it to be Borel. 

\begin{definition}\label{IntegratedRepn}
For each covariant representation $(\pi, u)$ of a twisted Banach algebra dynamical system  $(G, A, \alpha, \sigma)$ on a Banach space $E$, its  \textit{integrated representation} $\pi \rtimes u \colon L^1(G, A, \alpha, \sigma) \to \Li(E)$ is defined, for each $f \in L^1(G, A, \alpha, \sigma)$, by
\[
(\pi \rtimes u ) (f) \coloneqq \int_G \pi(f(x))u_x dx. 
\]
\end{definition}

A more general version of Lemma 2.8 in \cite{Bardadyn_2024}
also holds in the twisted case:
\begin{corollary}\label{1-1corresp}
Covariant representations of a twisted Banach algebra dynamical system
$(G, A, \alpha, \sigma)$ are in one-to-one correspondence with nondegenerate representations of $L^1(G, A, {\alpha, \sigma})$ on Banach spaces via the map $(\pi,u) \mapsto \pi \rtimes u$.
\end{corollary}
\begin{proof}
$L^1(G, A, \alpha, \sigma)$ has a cai thanks to Proposition \ref{BAIL1}. The rest of the proof is carried out exactly as parts (1) and (2) of Theorem 3.3 in \cite{BusSmi70}, 
by simply replacing every instance of `unitary' by  `invertible isometry'.  
\end{proof}

We are now ready to define twisted Banach crossed products, much in the spirit of \cite[Definitions 3.1 and 3.2]{CPBAI_11}. 

\begin{definition}\label{UnifBound}
Let $(G, A, \alpha, \sigma)$ be a twisted Banach algebra dynamical system. 
We say a class $\mathcal{R}$ of covariant representations of $(G, A, \alpha, \sigma)$
is \textit{uniformly bounded} if 
\begin{equation}\label{UnifBoundPI}
C_\mathcal{R} \coloneqq \sup \{ \|\pi\| \colon (\pi, u) \in \mathcal{R}\}< \infty.
\end{equation}
\end{definition}

\begin{definition}\label{def:R-tw-crossed-pr}
Let $(G, A, \alpha, \sigma)$ be a twisted Banach algebra dynamical system and let 
$\mathcal{R}$ be a nonempty uniformly bounded class of covariant representations of $(G, A, \alpha, \sigma)$. 
The \textit{twisted Banach crossed product of  $(G, A, \alpha, \sigma)$
associated with $\mathcal{R}$}, denoted 
 by $F_{\mathcal{R}}(G, A, \alpha, \sigma)$,
is defined as the Hausdorff completion of $L^1(G, A, \alpha, \sigma)$
with respect to the seminorm 
\[
\| f \|_{\mathcal{R}} \coloneqq \sup\{ \| (\pi \rtimes u ) (f) \| \colon (\pi, u) \in \mathcal{R} \}.
\]
More concretely,  $F_{\mathcal{R}}(G, A, \alpha, \sigma)$ is the completion 
of $L^1(G, A, \alpha, \sigma) / \ker(\| - \|_{\mathcal{R}} )$ under the norm 
induced by $\| -\|_{\mathcal{R}}$.
\end{definition}

\begin{remark}\label{Facts:F_R}
The following list consists of facts and notational conventions regarding $F_{\mathcal{R}}(G, A, \alpha, \sigma)$. All the facts are 
standard to prove and will be often used in this paper. In what follows 
$f \in  L^1(G, A, \alpha, \sigma)$:
\begin{enumerate}
\item The induced norm on $L^1(G, A, \alpha, \sigma) / \ker(\| - \|_{\mathcal{R}} )$ is still denoted by $\|-\|_\mathcal{R}$, so if $q_{\mathcal{R}} \colon L^1(G, A, \alpha, \sigma) \to L^1(G, A, \alpha, \sigma) / \ker(\| - \|_{\mathcal{R}} )$ is the quotient map, we then have $\| q_\mathcal{R}(f) \|_{\mathcal{R} }= \| f \|_{\mathcal{R}}$.
\item The norm and multiplication on the Banach algebra $F_{\mathcal{R}}(G, A, \alpha, \sigma)$ are still denoted by $\|-\|_\mathcal{R}$ and $*_{\alpha, \sigma}$ respectively. 
\item There is an isometric map $j_{\mathcal{R}} \colon L^1(G, A, \alpha, \sigma) / \ker(\| - \|_{\mathcal{R}} ) \to F_{\mathcal{R}}(G, A, \alpha, \sigma)$ that has dense range, and therefore the composition 
\begin{equation}\label{Tau_R}
\tau_{\mathcal{R} }\coloneqq j_\mathcal{R} \circ q_\mathcal{R} \colon L^1(G, A, \alpha, \sigma) \to F_{\mathcal{R}}(G, A, \alpha, \sigma)
\end{equation}
is a map with dense range satisfying $\| \tau_\mathcal{R}(f)\|_{\mathcal{R} }= \| f \|_{\mathcal{R}}$. 
\item \label{ILT->F} The space $\tau_\mathcal{R}(C_c(G,A, \alpha, \sigma))$ is a dense subalgebra of $F_{\mathcal{R}}(G, A, \alpha, \sigma)$. Moreover, if $(g_\lambda)_{\lambda \in \Lambda}$ is a net in $C_c(G,A, \alpha, \sigma)$ that converges in the inductive limit topology to $g \in C_c(G,A, \alpha, \sigma)$, then $\tau_\mathcal{R}(g_\lambda)$ converges to $\tau_{\mathcal{R}}(g)$ in $F_{\mathcal{R}}(G, A, \alpha, \sigma)$. 
\item If $(\pi, u) \in \mathcal{R}$ is a covariant representation of $(G,A,\alpha, \sigma)$ 
on a Banach space $E$, the map $\pi \rtimes u$ satisfies $\| (\pi \rtimes u)(f)\| \leq \| f \|_{\mathcal{R}}$ and therefore gives rise to a contractive representation $(\pi \rtimes u)^{\mathcal{R}} \colon F_{\mathcal{R}}(G, A, \alpha, \sigma) \to \Li(E)$ that satisfies 
\[
(\pi \rtimes u)^{\mathcal{R}}(\tau_\mathcal{R}(f)) = (\pi \rtimes u)(f).
\]
Furthermore, for any $c \in F_{\mathcal{R}}(G, A, \alpha, \sigma)$, we still have
\[
\| c \|_{\mathcal{R}} = \sup\{ \| (\pi \rtimes u )^\mathcal{R} (c) \| \colon (\pi, u) \in \mathcal{R} \}.
\]
\end{enumerate}

\end{remark}

\begin{remark}
If the twist $\sigma$ is trivial (i.e., $\sigma=\mathbf{1}$ where $\mathbf{1}_{x,y}=1_A \in M(A)$ for all $x,y \in G$) then we get $F_\mathcal{R}(G,A,\alpha, \mathbf{1}) = (A \rtimes_\alpha G)^\mathcal{R}$ from \cite[Definition 3.2]{CPBAI_11}. 
\end{remark}

\begin{theorem}\label{F_R:BAI}
Let $(G, A, \alpha, \sigma)$ be a twisted Banach algebra dynamical system and let 
$\mathcal{R}$ be a nonempty uniformly bounded class of covariant representations of $(G, A, \alpha, \sigma)$. Then $F_{\mathcal{R}}(G, A, \alpha, \sigma)$ has a $C_\mathcal{R}$-bounded approximate identity. 
\end{theorem}
\begin{proof}
Let $(f_{\lambda, U})_{\lambda \in \Lambda, U \in \mathcal{U}}$ be the cai in $L^1(G,A,\alpha, \sigma)$ as written in Remark \ref{ExplicitCAI}. Then, 
for any $(\pi, u) \in \mathcal{R}$, we have
\[
\|(\pi \rtimes u) (f_{\lambda,U})   \| \leq \| \pi \| \int_G  \psi_U(x)  dx \leq C_\mathcal{R}.
\]
Hence, if $\tau_{\mathcal{R}}$ is the map defined in Equation \eqref{Tau_R}, we 
get 
\[
\| \tau_{\mathcal{R}}(f_{\lambda, U}) \|_{\mathcal{R}} = \| f_{\lambda, U} \|_{\mathcal{R}}  \leq C_\mathcal{R}.
\]
Finally, by Proposition \ref{BAIL1}, $(f_{\lambda, U})_{\lambda \in \Lambda, U \in \mathcal{U}}$ is a cai for $C_c(G,A,\alpha, \sigma)$ equipped with the inductive limit topology, 
so Part \eqref{ILT->F} in Remark \ref{Facts:F_R} implies that the net $(\tau_{\mathcal{R}}(f_{\lambda, U}))_{\lambda \in \Lambda, U \in \mathcal{U}}$ is a $C_\mathcal{R}$-approximate identity for $F_{\mathcal{R}}(G, A, \alpha, \sigma)$, as wanted. 
\end{proof}

\section{Universal Properties}\label{S:U_P}

In what follows we let $\mathcal{R}$ be a fixed nonempty and uniformly bounded class of covariant representations 
of a twisted Banach algebra dynamical system $(G, A, \alpha, \sigma)$. Our main goal in this section is to define a ``universal'' covariant representation of $(G,A, \alpha, \sigma)$ on the Banach space $F_{\mathcal{R}}(G,A, \alpha, \sigma)$. In order to do so, we need the extra assumption that $\mathcal{R}$ consists only of contractive representations. This is only needed to faithfully represent $M(F_{\mathcal{R}}(G,A, \alpha, \sigma))$ on $F_{\mathcal{R}}(G,A, \alpha, \sigma)$, so we will introduce the extra assumption of $C_\mathcal{R} \leq 1$ right before is needed. 

We start by defining
a canonical map  $A  \to M(F_{\mathcal{R}}(G,A, \alpha, \sigma))$.
This will be done in several steps. 

\begin{definition}\label{AtoM(L^1)}
For each $a \in A$, define $\lambda_A(a) \colon L^1(G,A, \alpha, \sigma) \to L^1(G,A, \alpha, \sigma)$, 
for each $f \in L^1(G,A,\alpha,\sigma)$ and $x\in G$, 
by
\[
(\lambda_A(a)f)(x) \coloneqq af(x).
\]
Next, for $a \in A$ we define $\rho_A(a) \colon L^1(G,A, \alpha, \sigma) \to L^1(G,A, \alpha, \sigma)$,
for any $f \in L^1(G,A,\alpha,\sigma)$ and $x\in G$, 
as
\[
 (\rho_A(a)f)(x) \coloneqq f(x)\alpha_x(a).
\]
\end{definition}

\begin{lemma}\label{propAtoM(L^1)}
Let $f, g \in  L^1(G,A,\alpha,\sigma)$, let $a \in A$, let $(\pi, u) \in \mathcal{R}$, and let $C_{\mathcal{R}}$ be as in 
Equation \eqref{UnifBoundPI}. Then 
\begin{enumerate}
\item $\lambda_A(a)(f*_{\alpha, \sigma} g ) = (\lambda_A(a)f) *_{\alpha, \sigma} g$, \label{ALeftC1}
\item $(\pi \rtimes u)(\lambda_A(a)f) = \pi(a) (\pi \rtimes u)(f)$, \label{ALeftC2}
\item $\| \lambda_A(a) f\|_\mathcal{R} \leq C_{\mathcal{R}} \|a\|\| f \|_{\mathcal{R}}$, \label{ALeftC3}
\item $\rho_A(a)(f*_{\alpha, \sigma} g ) = f*_{\alpha, \sigma} (\rho_A(a)g)$,\label{ARightC1}
\item $(\pi \rtimes u)(\rho_A(a)f)  =  (\pi \rtimes u)(f)\pi(a)$, \label{ARightC2}
\item $\| \rho_A(a) f\|_\mathcal{R} \leq C_{\mathcal{R}} \|a\|\| f \|_{\mathcal{R}}$, \label{ARightC3}
\item $f*_{\alpha, \sigma} (\lambda_A(a)g) =  (\rho_A(a)f) *_{\alpha, \sigma} g$. \label{A_MULT}
\end{enumerate}
\end{lemma}
\begin{proof}
A direct computation shows that $a(f*_{\alpha, \sigma} g)(x)=((\lambda_A(a)f) *_{\alpha, \sigma} g)(x)$ 
for any $x \in G$, so Part \eqref{ALeftC1} follows. Next, 
\[
(\pi \rtimes u)(\lambda_A(a)f) = \int_G \pi(a)\pi(f(x))u_x dx = \pi(a) (\pi \rtimes u)(f),
\]
proving Part \eqref{ALeftC2}. Part \eqref{ALeftC3} is now a direct consequence of Part \eqref{ALeftC2}
(exactly as in Lemma 6.3 of \cite{CPBAI_11}). 

For Part \eqref{ARightC1}, take any $x \in G$ and using Condition \eqref{x,yAd} from Definition \ref{TwistedBanachSystem} in the second step we get
\begin{align*}
[\rho_A(a)(f*_{\alpha, \sigma} g)](x) & = \int_G f(y)\alpha_y(g(y^{-1}x))\sigma(y, y^{-1}x) \alpha_x(a)  dy\\
& =  \int_G f(y)\alpha_y(g(y^{-1}x))\alpha_y(\alpha_{y^{-1}x}(a))  \sigma(y, y^{-1}x) dy\\
& =  \int_G f(y)\alpha_y(g(y^{-1}x)\alpha_{y^{-1}x}(a))  \sigma(y, y^{-1}x) dy\\
& =  \int_G f(y)\alpha_y( [\rho_A(a)g](y^{-1}x) )  \sigma(y, y^{-1}x) dy\\
& = [f*_{\alpha, \sigma} (\rho_A(a)g)](x).
\end{align*}
Part \eqref{ARightC2} is a direct consequence of Condition \eqref{CovC4} in Definition \eqref{CovariantRep}, and Part  \eqref{ARightC3} follows in turn from  Part \eqref{ARightC2}, 
see also Proposition 6.5 in \cite{CPBAI_11}. 

Finally, to show Part \eqref{A_MULT}, we take any $x \in G$ and compute
\begin{align*}
[f*_{\alpha, \sigma} (\lambda_A(a)g)](x) & =  \int_G f(y)\alpha_y( [\lambda_A(a)g](y^{-1}x) )  \sigma(y, y^{-1}x) dy\\
& =  \int_G f(y)\alpha_y(a)\alpha_y(g(y^{-1}x) )  \sigma(y, y^{-1}x) dy\\
& =  \int_G [\rho_A(a)f](y) \alpha_y(g(y^{-1}x) )  \sigma(y, y^{-1}x) dy\\
& =[(\rho_A(a)f)*_{\alpha, \sigma} g](x),
\end{align*}
so we are done. 
\end{proof}

Observe now that Parts \eqref{ALeftC3} and \eqref{ARightC3} in the previous Lemma 
show that, for each $a \in A$,  the maps $\lambda_A(a)$ and $\rho_A(a)$ give rise 
to bounded operators $\lambda^{\mathcal{R}}_A(a),\rho^{\mathcal{R}}_A(a) \in \Li(F_{\mathcal{R}}(G, A, \alpha, \sigma))$ such that if $\tau_{\mathcal{R}} \colon L^1(G,A, \alpha, \sigma) \to F_{\mathcal{R}}(G, A, \alpha, \sigma)$ is the map from \eqref{Tau_R}, then 
\begin{align*}
\lambda^{\mathcal{R}}_A(a) \circ \tau_{\mathcal{R}} & = \tau_{\mathcal{R}}  \circ \lambda_A(a), 
& \| \lambda^{\mathcal{R}}_A(a) \| & \leq C_{\mathcal{R}} \|a\|,\\
\rho^{\mathcal{R}}_A(a) \circ \tau_{\mathcal{R}} & =  \tau_{\mathcal{R}}  \circ\rho_A(a),
& \| \rho^{\mathcal{R}}_A(a) \| & \leq C_{\mathcal{R}} \|a\|.
\end{align*}
Further, Parts \eqref{ALeftC1}, \eqref{ARightC1}, and \eqref{A_MULT}  in Lemma \ref{propAtoM(L^1)} give the desired canonical map $(\lambda^{\mathcal{R}}_A, \rho^{\mathcal{R}}_A) \colon A \to M(F_{\mathcal{R}}(G, A, \alpha, \sigma))$ by letting for each $a \in A$ 
\begin{equation}\label{A^RtoM}
(\lambda^{\mathcal{R}}_A, \rho^{\mathcal{R}}_A)(a) \coloneqq (\lambda^{\mathcal{R}}_A(a), \rho^{\mathcal{R}}_A(a)).
\end{equation}
In a similar fashion, we will now define a canonical map $ G \to M(F_{\mathcal{R}}(G,A, \alpha, \sigma))$. 
It is worth mentioning that, in contrast to Definition \ref{AtoM(L^1)}, in which the pair $(\lambda_A, \rho_A)$ coincides with $(i_A, j_A)$ from the untwisted case (see Proposition 5.3. in \cite{CPBAI_11}), 
we will now need the twist $\sigma$ in order to define the pair $(\lambda_G, \rho_G)$.

\begin{definition}\label{GtoM(L^1)}
For each $y \in G$, define $\lambda_G(y) \colon L^1(G,A, \alpha, \sigma) \to L^1(G,A, \alpha, \sigma)$
by letting, for any $f \in L^1(G,A,\alpha,\sigma)$ and $x\in G$, 
\[
(\lambda_G(y)f)(x) \coloneqq \alpha_y(f(y^{-1}x))\sigma_{y,y^{-1}x}.
\]
Similarly, for $y \in G$ we define a map $\rho_G(y) \colon L^1(G,A, \alpha, \sigma) \to L^1(G,A, \alpha, \sigma)$ such that for any $f \in L^1(G,A,\alpha,\sigma)$ and $x\in G$, we put
\[
 (\rho_G(y)f)(x) \coloneqq f(xy^{-1})\sigma_{xy^{-1}, y}\Delta(y^{-1}).
\]
\end{definition}

\begin{lemma}\label{propGtoM(L^1)}
Let $f, g \in  L^1(G,A,\alpha,\sigma)$, let $y \in G$, and let $(\pi, u) \in \mathcal{R}$. Then 
\begin{enumerate}
\item $\lambda_G(y)(f*_{\alpha, \sigma} g ) = (\lambda_G(y)f) *_{\alpha, \sigma} g$,  \label{GLeftC1}
\item $(\pi \rtimes u)(\lambda_G(y)f)  = u_y (\pi \rtimes u)(f)$, \label{GLeftC2}
\item $\| \lambda_G(y) f\|_\mathcal{R} = \| f \|_{\mathcal{R}}$,\label{GLeftC3}
\item $\rho_G(y)(f*_{\alpha, \sigma} g ) = f*_{\alpha, \sigma} (\rho_G(y)g)$,\label{GRightC1}
\item $(\pi \rtimes u)(\rho_G(y)f)  =  (\pi \rtimes u)(f)u_y$,\label{GRightC2}
\item $\| \rho_G(y) f\|_\mathcal{R} = \| f \|_{\mathcal{R}}$,\label{GRightC3}
\item $f*_{\alpha, \sigma} (\lambda_G(y)g) =  (\rho_G(y)f) *_{\alpha, \sigma} g$. \label{G_MULT}
\end{enumerate}
\end{lemma}
\begin{proof}
Let $x \in G$, by definition we get on the one hand 
\[
[\lambda_G(y)(f*_{\alpha, \sigma} g )](x)=\alpha_y\left( \int_G f(z) \alpha_{z}(g(z^{-1}y^{-1}x))\sigma_{z,z^{-1}y^{-1}x} dz\right)\sigma_{y,y^{-1}x}.
\]
On the other hand, letting $z=y^{-1}w$ together with left invariance of Haar measure in the third step, 
and using Conditions \eqref{x,yAd} and \eqref{ASzxy} from Definition \ref{TwistedBanachSystem} in the fourth and fifth steps, respectively, we get
\begin{align*}
[(\lambda_G(y)f)*_{\alpha, \sigma} g )](x)
 & = \int_G (\lambda_G(y)f)(w) \alpha_w(g(w^{-1}x))\sigma_{w,w^{-1}x}dw\\
 & = \int_G \alpha_y(f(y^{-1}w))\sigma_{y,y^{-1}w} \alpha_w(g(w^{-1}x))\sigma_{w,w^{-1}x}dw\\
 & =  \int_G \alpha_y(f(z))\sigma_{y,z} \alpha_{yz}(g(z^{-1}y^{-1}x))\sigma_{yz,z^{-1}y^{-1}x}dz\\
  & =  \int_G \alpha_y(f(z))\alpha_{y}(\alpha_z(g(z^{-1}y^{-1}x)))\sigma_{y,z} \sigma_{yz,z^{-1}y^{-1}x}dz\\
    & =  \int_G \alpha_y(f(z))\alpha_{y}(\alpha_z(g(z^{-1}y^{-1}x))\sigma_{z,z^{-1}y^{-1}x}) \sigma_{y,y^{-1}x}dz\\
    & = \alpha_y \left(\int_G f(z)\alpha_z(g(z^{-1}y^{-1}x))\sigma_{z,z^{-1}y^{-1}x} dz\right) \sigma_{y,y^{-1}x}.
\end{align*}
This proves that $\lambda_G(y)(f*_{\alpha, \sigma} g )=(\lambda_G(y)f)*_{\alpha, \sigma} g $, so Part \eqref{GLeftC1} is done. Next, by letting $z=y^{-1}x$ in the second step and using Conditions \eqref{CovC3} and \eqref{CovC4} from Definition \ref{CovariantRep} in the third and fourth steps, respectively,
 we compute 
\begin{align*}
(\pi \rtimes u)(\lambda_G(y)f ) 
& = \int_G \pi(\alpha_y(f(y^{-1}x))\sigma_{y,y^{-1}x}) u_x dx \\
& = \int_G \pi(\alpha_y(f(z))\sigma_{y,z}) u_{yz} dz \\
& = \int_G \pi(\alpha_y(f(z))) u_{y}u_{z} dz \\
& = \int_G u_{y} \pi(f(z))u_{z} dz \\
& = u_{y}(\pi \rtimes u)(f).
\end{align*}
This proves Part \eqref{GLeftC2} and Part \eqref{GLeftC3} follows as a direct consequence
using that $u_y$ is an invertible isometry. 

For Part \eqref{GRightC1}, let $x \in G$, and using Condition \eqref{ASzxy} from Definition \ref{TwistedBanachSystem} in the third step we obtain 
\begin{align*}
[\rho_G(y)(f*_{\alpha, \sigma} g )](x) & = (f*_{\alpha, \sigma} g)(xy^{-1})\sigma_{xy^{-1}, y}\Delta(y^{-1})\\
& = \int_G f(z)\alpha_z(g(z^{-1}xy^{-1}))\sigma_{z,z^{-1}xy^{-1}}\sigma_{xy^{-1},y}\Delta(y^{-1})dz\\
& = \int_G f(z)\alpha_z(g(z^{-1}xy^{-1}) \sigma_{z^{-1}xy^{-1},y}\Delta(y^{-1}))\sigma_{z,z^{-1}x}dz\\
&=\int_G f(z)\alpha_z([\rho_G(y)g](z^{-1}x))\sigma_{z,z^{-1}x}dz\\
& = [f*_{\alpha, \sigma} (\rho_G(y)g )](x).
\end{align*}
Next, observe that letting $z=xy^{-1}$, so that $dz=\Delta(y^{-1})dx$, 
in the third step, and using Condition \eqref{CovC3} from Definition \ref{CovariantRep}, 
at the forth one, gives
\begin{align*}
(\pi \rtimes u)(\rho_G(y)f)  & = \int_G \pi([\rho_G(y)f)](x))u_x dx \\
& = \int_G \pi(f(xy^{-1})\sigma_{xy^{-1}, y}\Delta(y^{-1}))u_x dx \\
& = \int_G \pi(f(z)\sigma_{z, y})u_{zy} dz \\
& = \int_G \pi(f(z))u_{z}u_{y} dz \\
& = (\pi \rtimes u)(f)u_y,
\end{align*}
proving Part \eqref{ARightC2}, and as before Part \eqref{ARightC3} now also follows using that $u_y$ is an invertible isometry. 

Finally, to show Part \eqref{G_MULT}, let $x \in G$, 
and using Conditions \eqref{x,yAd} and \eqref{ASzxy} from Definition \ref{TwistedBanachSystem} in the fourth step and fifth steps, respectively, and then  
letting $w=zy$, so that $dz=\Delta(y^{-1})dw$, in the sixth step, we find 
\begin{align*}
[f*_{\alpha, \sigma} (\lambda_G(y)g)](x) & = \int_G f(z) \alpha_z([\lambda_G(y)g](z^{-1}x))\sigma_{z,z^{-1}x}dz\\
& = \int_G f(z) \alpha_z(\alpha_{y}(g(y^{-1}z^{-1}x))\sigma_{y,y^{-1}z^{-1}x})\sigma_{z,z^{-1}x}dz\\
& = \int_G f(z) \alpha_z(\alpha_{y}(g(y^{-1}z^{-1}x))) \sigma_{z,y}\sigma_{zy, y^{-1}z^{-1}x} dz\\
& = \int_G f(z) \sigma_{z,y}\alpha_{zy}(g(y^{-1}z^{-1}x)) \sigma_{zy, y^{-1}z^{-1}x} dz\\
& = \int_G f(wy^{-1}) \sigma_{wy^{-1},y}\Delta(y^{-1})\alpha_{w}(g(w^{-1}x)) \sigma_{w, w^{-1}x} dw\\
& = \int_G (\rho_G(y)f)(w)\alpha_{w}(g(w^{-1}x)) \sigma_{w, w^{-1}x} dw\\
& =[(\rho_G(y)f) *_{\alpha, \sigma} g](x),
\end{align*}
finishing the proof. 
\end{proof}

As with $(\lambda_A, \rho_A)$, Parts \eqref{GLeftC3} and \eqref{GRightC3} in the previous Lemma 
show that, for each $y \in G$,  the maps $\lambda_G(y)$ and $\rho_G(y)$ give rise 
to bounded isometries $\lambda^{\mathcal{R}}_G(y),\rho^{\mathcal{R}}_G(y) \in \Li(F_{\mathcal{R}}(G, A, \alpha, \sigma))$ such that if $\tau_{\mathcal{R}} \colon L^1(G,A, \alpha, \sigma) \to F_{\mathcal{R}}(G, A, \alpha, \sigma)$ is the map from \eqref{Tau_R}, then 
\begin{align*}
\lambda^{\mathcal{R}}_G(y) \circ \tau_{\mathcal{R}} & = \tau_{\mathcal{R}}  \circ \lambda_G(y), 
& \| \lambda^{\mathcal{R}}_G(y) \| & =1,\\
\rho^{\mathcal{R}}_G(y) \circ \tau_{\mathcal{R}} & =  \tau_{\mathcal{R}}  \circ\rho_G(y),
& \| \rho^{\mathcal{R}}_G(y) \| & = 1.
\end{align*}
Further, Parts \eqref{GLeftC1}, \eqref{GRightC1}, and \eqref{G_MULT}  in Lemma \ref{propGtoM(L^1)} give the desired canonical map $(\lambda^{\mathcal{R}}_G, \rho^{\mathcal{R}}_G) \colon G \to M(F_{\mathcal{R}}(G, A, \alpha, \sigma))$ by letting for each $y \in G$ 
\begin{equation}\label{G^RtoM}
(\lambda^{\mathcal{R}}_G, \rho^{\mathcal{R}}_G)(y) \coloneqq (\lambda^{\mathcal{R}}_G(y), \rho^{\mathcal{R}}_G(y)).
\end{equation}

The isometries $\lambda^{\mathcal{R}}_G(y)$ and $\rho^{\mathcal{R}}_G(y)$ are in fact invertible ones. 
Since we only need the explicit inverse of $\lambda^{\mathcal{R}}_G(y)$, we present 
the details below and omit the analogous computations for $\rho^{\mathcal{R}}_G(y)$. 

\begin{definition}\label{L_G-inverse}
For each $y \in G$, define $\kappa_G(y) \colon L^1(G,A, \alpha, \sigma) \to L^1(G,A, \alpha, \sigma)$
for $f \in L^1(G,A, \alpha, \sigma)$ and $x \in G$, as 
\[
(\kappa_G(y)f)(x) \coloneqq \alpha_y^{-1}(f(yx)\sigma_{yx,x^{-1}})\sigma_{x,x^{-1}}^{-1}.
\]
\end{definition}

\begin{lemma}\label{L_G-inversePf}
Let $f \in L^1(G, A,\alpha, \sigma)$, let $y \in G$, and let $(\pi, u) \in \mathcal{R}$. Then
\begin{enumerate}
\item $\lambda_G(y) \circ \kappa_G(y) = \kappa_G(y) \circ \lambda_G(y)  = \op{id}_{ L^1(G, A,\alpha, \sigma)}$, \label{kappa_inv}
\item $(\pi \rtimes u)(\kappa_G(y)f)=u_y^{-1}(\pi \rtimes u)(f)$, \label{kaappa_if}
\item $\| \kappa_G(y)f\|_{\mathcal{R}} = \| f \|_{\mathcal{R}}$. \label{kappa_R_norm}
\end{enumerate}
In particular, $\kappa_G(y)$ gives rise to a bounded isometry $\kappa^{\mathcal{R}}_G(y) \in \Li(F_{\mathcal{R}}(G, A, \alpha, \sigma))$ that is 
the inverse of $\lambda_G^{\mathcal{R}}(y)$. 
\end{lemma}
\begin{proof}
For $x,y \in G$ notice that  Conditions \eqref{ASzxy} and \eqref{sigma_1=1} in Definition \ref{TwistedBanachSystem} give at once that 
\[
\alpha_y(\sigma_{y^{-1}x,x^{-1}y}^{-1} ) = \sigma_{x,x^{-1}y}^{-1}\sigma_{y,y^{-1}x}^{-1}.
\]
Thus, using this last equality in the forth step below gives 
\begin{align*}
((\lambda_G(y) \circ \kappa_G(y))f)(x) & = \alpha_y(\kappa_G(y)(y^{-1}x)) \sigma_{y,y^{-1}x}\\
& = \alpha_y(     \alpha_y^{-1}(f(x)\sigma_{x,x^{-1}y})\sigma_{y^{-1}x,x^{-1}y}^{-1}      )       \sigma_{y,y^{-1}x}\\
& =    f(x)\sigma_{x,x^{-1}y}  \alpha_y(  \sigma_{y^{-1}x,x^{-1}y}^{-1}      )       \sigma_{y,y^{-1}x}\\
& =    f(x). 
\end{align*}
Similarly, for $x \in G$, we compute using  Conditions \eqref{ASzxy} and \eqref{sigma_1=1} in Definition \ref{TwistedBanachSystem} 
in the fourth and fifth steps respectively 
\begin{align*}
((\kappa_G(y) \circ \lambda_G(y))f)(x) & = \alpha_y^{-1}(\lambda_G(y)(yx)\sigma_{yx,x^{-1}})\sigma_{x,x^{-1}}^{-1}\\
& = \alpha_y^{-1}(\alpha_y(f(x))\sigma_{y,x}\sigma_{yx,x^{-1}})\sigma_{x,x^{-1}}^{-1}\\
& = f(x)\alpha_y^{-1}(\sigma_{y,x}\sigma_{yx,x^{-1}})\sigma_{x,x^{-1}}^{-1}\\
& = f(x)\alpha_y^{-1}(\alpha_y(\sigma_{x,x^{-1}}\sigma_{y,1_G}))\sigma_{x,x^{-1}}^{-1}\\
& = f(x).
\end{align*}
This proves Part \eqref{kappa_inv}. Next, 
for any $y, z \in G$, 
Conditions \eqref{CovC3} and \eqref{CovC4} in Definition \ref{CovariantRep}
together with Condition \eqref{sigma_1=1} in Definition \ref{TwistedBanachSystem} and the fact that $\widehat{\pi}$ is a unital representation 
of $M(A)$ show that 
\begin{equation}\label{TecnicalU}
 u_y \pi(  \alpha_y^{-1}(\sigma_{z,z^{-1}y})\sigma_{y^{-1}z,z^{-1}y}^{-1}   )u_{y^{-1}z} = u_z.
\end{equation}
Therefore, letting $z=yx$ in the second step, 
using Condition \eqref{CovC4} from Definition \ref{CovariantRep}
in the third step, and Equation \eqref{TecnicalU} in the fourth step we get 
\begin{align*}
(\pi \rtimes u)(\kappa_G(y)f) & = 
\int_G \pi(  \alpha_y^{-1}(f(yx)\sigma_{yx,x^{-1}})\sigma_{x,x^{-1}}^{-1}   )u_x dx \\
& = \int_G \pi(  \alpha_y^{-1}(f(z)\sigma_{z,z^{-1}y})\sigma_{y^{-1}z,z^{-1}y}^{-1}   )u_{y^{-1}z} dz \\
& = \int_G u_{y}^{-1}\pi(f(z))  u_y \pi(  \alpha_y^{-1}(\sigma_{z,z^{-1}y})\sigma_{y^{-1}z,z^{-1}y}^{-1}   )u_{y^{-1}z} dz \\
& =u_{y}^{-1} \int_G \pi(f(z))  u_z dz \\
& = u_y^{-1}(\pi \rtimes u)(f).
\end{align*}
This proves Part \eqref{kaappa_if} and Part \eqref{kappa_R_norm} follows as a direct consequence, finishing the proof. 
\end{proof}

From now on we will assume that $C_\mathcal{R} \leq 1$, that is for each $(\pi, u) \in \mathcal{R}$ the map $\pi$ is a contractive representation. As a consequence, Theorem \ref{F_R:BAI} implies that $F_\mathcal{R}(A,G,\alpha, \sigma)$ has a cai and therefore 
we can and will isometrically identify $M(F_\mathcal{R}(A,G,\alpha, \sigma))$ with 
a closed subalgebra of $\Li(F_\mathcal{R}(A,G,\alpha, \sigma))$ via the map $(L,R) \mapsto L$. 

\begin{proposition}
The map $(\lambda^{\mathcal{R}}_A, \rho_A^{\mathcal{R}}) \colon A \to M(F_{\mathcal{R}}(G,A,\alpha, \sigma))$ from Equation \eqref{A^RtoM}  together with the map $(\lambda^{\mathcal{R}}_G, \rho_G^{\mathcal{R}}) \colon G \to M(F_{\mathcal{R}}(G,A,\alpha, \sigma))$ from Equation \eqref{G^RtoM} form a covariant representations of $(G,A,\alpha, \sigma)$ on the Banach space $F_{\mathcal{R}}(G,A,\alpha, \sigma)$. That is $(\lambda^{\mathcal{R}}_A, \rho_A^{\mathcal{R}}) $ is a nondegenerate representation of $A$ on $F_{\mathcal{R}}(G,A,\alpha, \sigma)$, $(\lambda^{\mathcal{R}}_G, \rho_G^{\mathcal{R}})$ is a strongly continuous map $G \to \op{Iso}(F_{\mathcal{R}}(G,A,\alpha, \sigma))$, and 
for all $a \in A$, $x, y \in G$ we have 
\begin{align*}
(\lambda^{\mathcal{R}}_G, \rho_G^{\mathcal{R}})(x) 
(\lambda^{\mathcal{R}}_G, \rho_G^{\mathcal{R}})(y) & =  \widehat{(\lambda^{\mathcal{R}}_A, \rho_A^{\mathcal{R}})}(\sigma_{x,y}) (\lambda^{\mathcal{R}}_G, \rho_G^{\mathcal{R}})(xy),\\
(\lambda^{\mathcal{R}}_A, \rho_A^{\mathcal{R}})(\alpha_x(a)) & =(\lambda^{\mathcal{R}}_G, \rho_G^{\mathcal{R}})(x) (\lambda^{\mathcal{R}}_A, \rho_A^{\mathcal{R}})(a) 
(\lambda^{\mathcal{R}}_G, \rho_G^{\mathcal{R}})(x)^{-1}.
\end{align*}
\end{proposition}
\begin{proof}
Since $A$ has a cai, the nondegeneracy and strong continuity claims follow as in the untwisted case (see for instance \cite[Propositions 6.4 and 6.5]{CPBAI_11}). The fact that the image of $(\lambda^{\mathcal{R}}_G, \rho_G^{\mathcal{R}})$ is in $\op{Iso}(F_{\mathcal{R}}(G,A,\alpha, \sigma))$ follows at once 
from Lemma \ref{L_G-inversePf} and the discussion preceding the statement of this Proposition. 

Now it only suffices to show that, for any $x, y \in G$, the following two equations hold
\begin{align}\label{LR_Cov_1}
\lambda_G(x)\lambda_G(y) & = \widehat{\lambda_A}(\sigma_{x,y}),\lambda_G(xy) \\
\rho_G(y)\rho_G(x) & =\rho_G(xy) \widehat{\rho_A}(\sigma_{x,y}),\notag
\end{align}
 and that if in addition $a \in A$, then 
\begin{align}\label{LR_Cov_2}
\lambda_A(\alpha_x(a)) & = \lambda_G(x)\lambda_A(a)\lambda_G(x)^{-1}\\
\rho_A(\alpha_x(a)) & = \rho_G(x)^{-1}\rho_A(a)\rho_G(x),\notag
\end{align}
also hold. Indeed, for any $f \in L^1(G, A, \alpha, \sigma)$ and $z \in G$,
 using Conditions \eqref{ASzxy} and \eqref{x,yAd} from Definition \ref{TwistedBanachSystem} in the second and third steps, respectively, we get
\begin{align*}
[\lambda_G(x)\lambda_G(y)f](z) & = \alpha_x(\alpha_y(f(y^{-1}x^{-1}z))\sigma_{y,y^{-1}x^{-1}z})\sigma_{x,x^{-1}z} \\
& = \alpha_x(\alpha_y(f(y^{-1}x^{-1}z)) ) \sigma_{x,y}\sigma_{xy,x^{-1}z} \\
& = \sigma_{x,y}\alpha_{xy}(f(y^{-1}x^{-1}z)) \sigma_{xy,x^{-1}z} \\
& = [\widehat{\lambda}_A(\sigma_{x,y})\lambda_G(xy)f](z),
\end{align*}
which gives the first equality in Equation \eqref{LR_Cov_1}. The second one is proved similarly and therefore its proof is omitted. We end by showing the second equality in Equation \eqref{LR_Cov_2}, omitting this time 
the details for the proof of the first one. Indeed, if $f \in L^1(G, A, \alpha, \sigma)$ and $y \in G$, 
using Condition \ref{x,yAd} in Definition \ref{TwistedBanachSystem} in the second step, 
we get
\begin{align*}
[\rho_G(x)\rho_A(\alpha_x(a))f](y) & = f(yx^{-1})\alpha_{yx^{-1}}(\alpha_x(a))\sigma_{yx^{-1},x}\Delta(x^{-1}) \\
 & = f(yx^{-1})\sigma_{yx^{-1},x}\alpha_y(a)\Delta(x^{-1}) \\
  & = f(yx^{-1})\sigma_{yx^{-1},x}\Delta(x^{-1}) \alpha_y(a)\\
& =[\rho_A(a)\rho_G(x)f](y),
\end{align*}
so we are done. 
\end{proof}

Henceforth, we will often need the map $\iota_A \colon A \to M(A)$ from Equation \eqref{iota_A_map}.

\begin{proposition}\label{Universal_IntMap_Isom}
The space  $((\lambda_A^\mathcal{R}, \rho_A^\mathcal{R})  \rtimes (\lambda_G^\mathcal{R}, \rho_G^\mathcal{R}))(L^1(G,A, \alpha, \sigma))$ is a dense subset of $\iota_{F^{\mathcal{R}}(G,A,\alpha, \sigma)}(F^{\mathcal{R}}(G,A,\alpha, \sigma)) \subseteq M(F^{\mathcal{R}}(G,A,\alpha, \sigma))$.
Moreover, for any $f \in L^1(G,A, \alpha, \sigma)$, we have
\[
\| ((\lambda_A^\mathcal{R}, \rho_A^\mathcal{R})  \rtimes (\lambda_G^\mathcal{R}, \rho_G^\mathcal{R}))f\|=\| f \|_{\mathcal{R}}.
\]
\end{proposition}
\begin{proof}
Since $\tau_\mathcal{R}((L^1(G,A, \alpha, \sigma))$ is dense in $F^{\mathcal{R}}(G,A,\alpha, \sigma))$, for the density claim it suffices to show that if $f \in L^1(G,A, \alpha, \sigma)$, then
\[
((\lambda_A^\mathcal{R}, \rho_A^\mathcal{R})  \rtimes (\lambda_G^\mathcal{R}, \rho_G^\mathcal{R}))f \in (\iota_{F^{\mathcal{R}}(G,A,\alpha, \sigma)} \circ \tau_\mathcal{R})(L^1(G,A, \alpha, \sigma)).
\]  
A direct computation shows that for any $g \in  L^1(G,A, \alpha, \sigma)$ we have 
\[
((\lambda_A^\mathcal{R}, \rho_A^\mathcal{R})  \rtimes (\lambda_G^\mathcal{R}, \rho_G^\mathcal{R}))f)\tau_\mathcal{R}(g) = \tau_\mathcal{R}(f *_{\alpha, \sigma} g) = \iota_{F^{\mathcal{R}}(G,A,\alpha, \sigma)}( \tau_\mathcal{R}(f))\tau_\mathcal{R}(g).
\]
This gives at once that $((\lambda_A^\mathcal{R}, \rho_A^\mathcal{R})  \rtimes (\lambda_G^\mathcal{R}, \rho_G^\mathcal{R}))f = (\iota_{F^{\mathcal{R}}(G,A,\alpha, \sigma)} \circ \tau_\mathcal{R})f$ on all 
of $F^{\mathcal{R}}(G,A,\alpha, \sigma)$, so the desired density follows. Finally, since both $\iota_{F^{\mathcal{R}}(G,A,\alpha, \sigma)}$ and $\tau_\mathcal{R}$ are isometric, we get 
\[
\| ((\lambda_A^\mathcal{R}, \rho_A^\mathcal{R})  \rtimes (\lambda_G^\mathcal{R}, \rho_G^\mathcal{R}))f\|=\| f \|_{\mathcal{R}},
\]
as wanted. 
\end{proof}

\begin{definition}
We say that a covariant representation $(\pi,u)$ of $(G, A, \alpha, \sigma)$ on a Banach space $E$ is 
\textit{$\mathcal{R}$-continuous} if there is $M>0$ such that 
\[
\| (\pi \rtimes u)f \| \leq M \| f \|_{\mathcal{R}}
\]
for all $f \in L^1(G,A,\alpha, \sigma)$.
\end{definition}

\begin{remark}
 Observe that the above definition implies that 
\begin{enumerate}
\item if $(\pi, u) \in \mathcal{R}$, then $(\pi,u)$ is trivially $\mathcal{R}$-continuous 
with $M=1$;
\item if $(\pi, u)$ is an $\mathcal{R}$-continuous covariant representation on a Banach space $E$, 
then the map $\pi \rtimes u$ extends to $(\pi \rtimes u)^\mathcal{R} \colon F_{\mathcal{R}}(G,A,\alpha, \sigma) \to \Li(E)$ satisfying 
\begin{align*}
(\pi \rtimes u)^\mathcal{R} \circ \tau_{\mathcal{R}} & = \pi \rtimes u,
& \| (\pi \rtimes u)^\mathcal{R} \| & \leq M.
\end{align*}
\end{enumerate}
\end{remark}

We are now ready to state our correspondence theorem, which can be compared with the untwisted left version given in Theorem 2.1 of \cite{CPBAII_13}. 

\begin{theorem}\label{Thm:Corresp_THM}
Let $(G,A,\alpha, \sigma)$ be a twisted Banach algebra dynamical system. The map $(\pi , u) \mapsto (\pi \rtimes u)^\mathcal{R}$ is a bijection from the space of $\mathcal{R}$-continuous covariant representation of $(G, A, \alpha, \sigma)$ on Banach spaces to the space of nondegenerate representations of $F_{\mathcal{R}}(G,A,\alpha, \sigma)$ on Banach spaces. More precisely, if $(\pi, u)$ is  $\mathcal{R}$-continuous, then
\begin{align*}
\pi &= \widehat{(\pi \rtimes u)^{\mathcal{\mathcal{R}}}} \circ (\lambda^{\mathcal{R}}_A, \rho^{\mathcal{R}}_A) \\
 u &=   \widehat{(\pi \rtimes u)^{\mathcal{\mathcal{R}}}}  \circ (\lambda^{\mathcal{R}}_G, \rho^{\mathcal{R}}_G).
\end{align*}
Conversely, if $\Phi $ is a nondegenerate representation of $F_{\mathcal{R}}(G,A,\alpha, \sigma)$ on
the Banach space $E$, 
$\Pi \coloneqq \widehat{\Phi}  \circ (\lambda^{\mathcal{R}}_A, \rho^{\mathcal{R}}_A)$, and 
$U \coloneqq \widehat{\Phi}  \circ (\lambda^{\mathcal{R}}_G, \rho^{\mathcal{R}}_G)$, 
then $(\Pi,U)$ is an $\mathcal{R}$-continuous covariant representation of $(G, A, \alpha, \sigma)$ on $E$ 
and $\Phi = (\Pi \rtimes U)^{\mathcal{R}}$. 
\end{theorem}
\begin{proof}
Given the equalities \eqref{ALeftC2}, \eqref{ARightC2} from Lemma \ref{propAtoM(L^1)}, 
and the  equalities \eqref{GLeftC2}, \eqref{GRightC2} from Lemma \ref{propGtoM(L^1)}, 
the proof will go exactly as in the untwisted case (see \cite[Theorem 2.1]{CPBAII_13} and  \cite[Proposition 7.1]{CPBAI_11} for the details). 
\end{proof}

The universal property for the twisted Banach algebra crossed product is the contents of the following theorem.
\begin{theorem}\label{Universal_R_Prop}
Let $(G, A, \alpha, \sigma)$ be a Banach algebra twisted dynamical system, where $A$ has a cai, let $\mathcal{R}$ be a class of contractive covariant representations of $(G, A,\alpha, \sigma)$, 
and let $B$ be a Banach algebra with a cai that is nondegenerately representable. Suppose there are 
maps
$k_A \colon A \to M(B)$ and $k_G \colon G \to M(B)$ satisfying 
\begin{enumerate}
\item  $(k_A, k_G)$ is a covariant representation of 
 $(G, A, \alpha, \sigma)$ on $B$, \label{c_o_v}
 \item $(k_A \rtimes k_G )(L^1(G,A, \alpha, \sigma))$ is a dense subset of $\iota_B(B) \subseteq M(B)$, \label{dens}
  \item $\|(k_A \rtimes k_G) f \|=\|f\|_{\mathcal{R}}$ for all $f \in L^1(G,A, \alpha, \sigma)$.\label{kISOM}
\end{enumerate}
Then there is a unique isometric Banach algebra isomorphism $\Psi \colon F_{\mathcal{R}}(G, A, \alpha, \sigma) \to B$ such that $\widehat{\Psi} \colon M(F_{\mathcal{R}}(G, A, \alpha, \sigma)) \to M(B)$ satisfies 
\begin{align*}
k_A = \widehat{\Psi} \circ (\lambda_A^{\mathcal{R}}, \rho_A^{\mathcal{R}})\\
k_G = \widehat{\Psi} \circ (\lambda_G^{\mathcal{R}}, \rho_G^{\mathcal{R}}).
\end{align*}
\end{theorem}
\begin{proof}
Assumptions \eqref{c_o_v} and \eqref{kISOM} give that $(k_A, k_G)$ is $\mathcal{R}$-continuous, 
whence we get a map $(k_A \rtimes k_G )^{\mathcal{R}} \colon F_{\mathcal{R}}(G,A,\alpha, \sigma) \to \Li(B)$. 
Define $\Psi_0 \colon \tau_{\mathcal{R}}(L^1(G,A, \alpha, \sigma)) \to B$ by
\[
\Psi_0 \coloneqq \iota_B^{-1} \circ (k_A \rtimes k_G )^{\mathcal{R}},
\]
where $\iota_B^{-1}$ is the inverse of the isometry $\iota_B \colon B \to \iota_B(B)$. 
Assumption \eqref{kISOM} implies that $\| \Psi_0 c \|_B = \|c\|$ for all $c \in \tau_{\mathcal{R}}(L^1(G,A, \alpha, \sigma)) $ and therefore $\Psi_0$ extends to an isometry $\Psi \colon F_{\mathcal{R}}(G, A, \alpha, \sigma) \to B$. Furthermore, Assumption \eqref{dens} implies that $\Psi_0$ has dense range and therefore 
$\Psi$ is in fact surjective. The equalities for $k_A$ and $k_G$ now follow from Theorem \ref{Thm:Corresp_THM}. 
\end{proof}

\section{Twisted $L^p$-operator crossed products}\label{S:T_LP_CP}

In this section, following \cite{ncp2012AC}, we restrict to second countable locally compact groups. 
This in turn allows us to apply measure theory results such as the Tonelli-Fubini Theorem, Hölder's inequality, and the Lebesgue Dominated Convergence Theorem to the Lebesgue-Bochner spaces $L^p(\mu; E)$ from Definition \ref{Lebegue-Bochner}.

Below, we let $A$ be a nondegenerate $L^p$-operator algebra with a cai and 
 $(G, A, \alpha, \sigma)$ be a twisted Banach algebra dynamical system. We refer to  $(G, A, \alpha, \sigma)$ 
 as an $L^p$-twisted dynamical system. We start by defining $F^p(G, A, \alpha, \sigma)$, the full $L^p$-twisted crossed product:

\begin{definition}\label{defn_FullLpCP}
Let $p \in [1,\infty)$ and let $\mathcal{R}^p=\mathcal{R}^p(G, A, \alpha, \sigma)$ be all the 
covariant, $\sigma$-finite, and contractive representations of $(G, A, \alpha, \sigma)$ on $L^p$-spaces. 
The \textit{full $L^p$-twisted crossed product} is 
\[
F^p(G, A, \alpha, \sigma) \coloneqq F_{\mathcal{R}^p}(G, A, \alpha, \sigma).
\]
\end{definition}
Notice that $F^p(G, A, \alpha, \sigma)$ is guaranteed to exist as long as $A$ is a nondegenerate $L^p$-operator algebra with a cai,
so that $\mathcal{R}^p \neq \varnothing$. 

\begin{remark}\label{Rmk_twist_1_ncp}
When $\sigma=\mathbf{1}$ is the trivial twist, then $F^p(G, A, \alpha, \mathbf{1})$
coincides with $F^p(G, A, \alpha)$ from \cite[Definition 3.3]{ncp2013CP}. 
\end{remark}

Moreover, for each $L^p$-twisted dynamical system we will also get a reduced version of the twisted crossed product: $F_{\op{r}}^p(G, A, \alpha, \sigma)$. 
For this we first need to 
show that any contractive nondegenerate representation $\pi_0$ of $A$ on an $L^p$-space 
induces a covariant representations of $(G, A, \alpha, \sigma)$. Indeed, 
let $\pi_0 \colon A \to \Li(L^p(\Omega,\mu))$ be a contractive and nondegenerate representation. 
Let 
\begin{equation}\label{1-tensor_p}
E \coloneqq  L^p(G, \nu_G) \otimes_p L^p(\Omega, \mu) \cong L^p(G; L^p(\Omega,\mu)) \cong L^p(G \times \Omega, \nu_G \times \mu).
\end{equation}
For any $a \in A$, $\xi \in C_c(G, L^p(\Omega, \mu)) \subseteq E$, and $x \in G$, we define $\pi \colon A \to \Li(E)$ so that 
\begin{equation}\label{2_pi}
(\pi(a)\xi)(x) \coloneqq \pi_0(\alpha_x^{-1}(a))(\xi(x)) \in L^p(\Omega, \mu),
\end{equation}
Next, for any $x,y \in G$ and $\xi \in C_c(G, L^p(\Omega, \mu)) \subseteq E$, 
we define $u  \colon G \to \Li(E)$ so that
\begin{equation}\label{3_v}
(u_y\xi)(x) \coloneqq \widehat{\pi_0}(\alpha_x^{-1}(\sigma_{y,y^{-1}x}))(\xi(y^{-1}x))\in L^p(\Omega, \mu).
\end{equation}
Next we prove that 
$(\pi, u)$ is indeed covariant according to Definition \ref{CovariantRep}. 
\begin{proposition}\label{IndReducedCovRepn}
Let $(G, A, \alpha, \sigma)$ be an $L^p$-twisted Banach algebra dynamical system, 
let $p \in [1,\infty)$, let $\pi_0 \colon A \to \Li(L^p(\Omega, \mu))$ 
be a contractive and nondegenerate representation, and let 
$E$, $\pi$, $u$ be as in Equations \eqref{1-tensor_p}, \eqref{2_pi}, and \eqref{3_v}. Then the pair $(\pi, v)$
is a contractive covariant representation of  $(G, A, \alpha, \sigma)$ on $E$.
\end{proposition}
\begin{proof}
It follows from \cite[Lemma 2.11]{ncp2013CP} that $\pi$ is indeed a contractive 
nondegenerate representation of $A$ on $E$.
It remains to check Conditions \eqref{CovC2}, \eqref{CovC3}, and \eqref{CovC4} in Definition \ref{CovariantRep}.  
For Condition \eqref{CovC2}, we first check that $u_y \in \op{Iso}(E)$ for any $y \in G$. 
Indeed, fix $y \in G$,
and using that $\pi_0$ is contractive, $\alpha$ isometric, and $\sigma_{x,y}$ has norm 1, 
we obtain, for any $\xi \in C_c(G, L^p(\Omega, \mu)) \subseteq E$,
\[
\| u_y \xi \|^p = \int_G \|  (u_y\xi)(x) \|_p^p dx \leq \int_G \| \xi(y^{-1}x)\|_p^p = \| \xi\|^p,
\]
 proving that $u_y$ is contractive. Next, for $\xi \in C_c(G, L^p(\Omega, \mu))$ and $x \in G$, 
 define $v \colon G \to \Li(E)$ by letting
\begin{equation}\label{u_yINV}
(v_y\xi)(x) \coloneqq \widehat{\pi_0}(\alpha^{-1}_{yx}(\sigma_{y,x}^{-1}))(\xi(yx))\in L^p(\Omega, \mu).
\end{equation}
An analogous computation 
as the one we did for $u_y$ shows that $v_y$ is also
contractive. Further a direct computation shows that, for any $\xi \in C_c(G, L^p(\Omega, \mu))$, 
\[
v_yu_y\xi = u_yv_y\xi = \xi.
\]
Thus, $u_y^{-1}=v_y$, whence $u_y$ is in fact an isometry and therefore $u_y$ is indeed in $\op{Iso}(E)$, so Condition \eqref{CovC2} now follows. 

For Condition \eqref{CovC3}, it suffices to verify it on elements of the dense space $C_c(G, L^p(\Omega, \mu))$. 
Take $x, y \in G$ and $\xi \in C_c(G, L^p(\Omega, \mu))$, and we compute, for any $z \in G$, 
\begin{align*}
(u_xu_y \xi)(z) & = \widehat{\pi_0}(\alpha_z^{-1}(\sigma_{x,x^{-1}z}))((u_y \xi)(x^{-1}z))\\
& = \widehat{\pi_0}(\alpha_z^{-1}(\sigma_{x,x^{-1}z})) \widehat{\pi_0}(\alpha^{-1}_{x^{-1}z}(\sigma_{y,y^{-1}x^{-1}z}))(\xi(y^{-1}x^{-1}z))\\
& = \widehat{\pi_0}(\alpha_z^{-1}(\sigma_{x,x^{-1}z})\alpha^{-1}_{x^{-1}z}(\sigma_{y,y^{-1}x^{-1}z}))(\xi(y^{-1}x^{-1}z)).
\end{align*}
For convenience, set for a moment $T \coloneqq \alpha^{-1}_{x^{-1}z}(\sigma_{y,y^{-1}x^{-1}z}) \in M(A)$.
Using Condition \eqref{x,yAd} in Definition \ref{TwistedBanachSystem} in the second step, 
then the definition of $T$ in the third step, and finally Condition \eqref{ASzxy} in Definition \ref{TwistedBanachSystem} in the fourth step, we get 
\begin{align*}
\alpha_z^{-1}(\sigma_{x,x^{-1}z})T& = \alpha_{z}^{-1}(\sigma_{x,x^{-1}z} \alpha_z(T))\\
&  = \alpha_{z}^{-1}(\alpha_x(\alpha_{x^{-1}z}(T))\sigma_{x,x^{-1}z} )\\
& =  \alpha_{z}^{-1}(\alpha_x(\sigma_{y,y^{-1}x^{-1}z})\sigma_{x,x^{-1}z} )\\
& =  \alpha_{z}^{-1}(\sigma_{x,y}\sigma_{xy, y^{-1}x^{-1}z} ).
\end{align*}
Therefore, going back to the initial computation for $(u_xu_y \xi)(z)$,
we now have 
\begin{align*}
(u_xu_y \xi)(z) & = \widehat{\pi_0}(\alpha_z^{-1}(\sigma_{x,x^{-1}z})T )(\xi(y^{-1}x^{-1}z)) \\
 & =  \widehat{\pi_0}( \alpha_{z}^{-1}(\sigma_{x,y}\sigma_{xy, y^{-1}x^{-1}z} ) )(\xi(y^{-1}x^{-1}z)). \\
 & = \widehat{\pi_0}( \alpha_{z}^{-1}(\sigma_{x,y}) )\widehat{\pi_0}( \alpha_{z}^{-1}(\sigma_{xy, y^{-1}x^{-1}z} ))(\xi(y^{-1}x^{-1}z))\\
  & = \widehat{\pi_0}( \alpha_{z}^{-1}(\sigma_{x,y}) )(u_{xy}\xi (z))\\
  & = (\widehat{\pi}(\sigma_{x,y})u_{xy}\xi)(z).
\end{align*}
 Condition \eqref{CovC3} now follows. Finally, to show Condition \eqref{CovC4}, 
 we also compute using elements in $C_c(G, L^p(\Omega, \mu))$. 
Indeed, for $x, y \in G$ $a \in A$ and $\xi \in C_c(G, L^p(\Omega, \mu))$, 
recalling that $u_x^{-1}=v_x$ (as in Equation \eqref{u_yINV})
in the first step and using Condition \eqref{x,yAd} from Definition \ref{TwistedBanachSystem}
in the fifth step, we get
 \begin{align*}
(u_x \pi(a) u_x^{-1} \xi)(y) & =  \widehat{\pi}_0(\alpha_y^{-1}(\sigma_{x,x^{-1}y}))(\pi(a)v_x\xi)(x^{-1}y)\\
 & =   \widehat{\pi}_0(\alpha_y^{-1}(\sigma_{x,x^{-1}y}))\pi_0(\alpha_{x^{-1}y}^{-1}(a))(v_x\xi(x^{-1}y))\\
 & =  \widehat{\pi}_0(\alpha_y^{-1}(\sigma_{x,x^{-1}y}))\pi_0(\alpha_{x^{-1}y}^{-1}(a))\widehat{\pi}_0(\alpha_y^{-1}(\sigma_{x,x^{-1}y}^{-1}))\xi(y)\\
  & =  \pi_0(\alpha_y^{-1}( 
  \sigma_{x,x^{-1}y} \alpha_z(\alpha_{x^{-1}y}^{-1}(a) )\sigma_{x,x^{-1}y}^{-1} 
  ))\xi(y)\\
    & =  \pi_0(\alpha_y^{-1}( 
\alpha_x(\alpha_{x^{-1}y}(\alpha_{x^{-1}y}^{-1}(a) ))
  ))\xi(y)\\  
      & =  \pi_0(\alpha_y^{-1}( 
\alpha_x(a)
  ))\xi(y)\\  
  & = (\pi(\alpha_x(a))\xi)(y),
 \end{align*}
 whence  Condition \eqref{CovC4} in now established and we are done. 
\end{proof}

\begin{definition}
Let $p \in [1, \infty)$ and let $\mathcal{R}_{\op{r}}^{p} \subseteq \mathcal{R}^p$
denote the set of covariant representations of $(G, A, \alpha, \sigma)$ induced by nondegenerate, $\sigma$-finite, and contractive representations of $A$ on $L^p$-spaces as in Proposition 
\ref{IndReducedCovRepn}.
The \textit{reduced $L^p$-twisted crossed product} is defined by
\[
F_{\op{r}}^p(G, A, \alpha, \sigma) \coloneqq F_{\mathcal{R}_{\op{r}}^{p}}(G, A, \alpha, \sigma).
\]
\end{definition}

\begin{remark}
When $A=\C$ and $G$ acts trivially on $\C$, then  $F_{\op{r}}^{p}(G, \C, \mathbf{1}, \sigma)$ is simply the algebra $F_\lambda^p(G, \sigma)$ from \cite[Definition 3.4]{HetOrt23}.  When $\sigma=\mathbf{1}$ is the trivial twist, then $F_{\op{r}}^p(G, A, \alpha, \mathbf{1})$
coincides with $F_{\op{r}}^p(G, A, \alpha)$ from \cite[Definition 3.3]{ncp2013CP}.
\end{remark}

\section{Packer-Raeburn trick: $p$-analogue}\label{S:P_R_T}

The following is a Banach algebra version of {\cite[Definition 3.1]{PacRae89}}.
\begin{definition}
Let $G$ be a locally compact group and let 
$A$ be a nondegenerate Banach algebra with a cai. We say that two twisted actions $(\alpha, \sigma)$ and $(\beta, \omega)$ of $G$ on $A$ 
are \textit{exterior equivalent} if there is a map
$\theta \colon G \to \op{Inv}_1(M(A))$ such that 
\begin{enumerate}
\item $\theta$ is strictly continuous, that is both $x \mapsto \theta_xa$ and 
$x \mapsto a\theta_x$ are continuous maps $G \to A$ for all $a \in A$, 
\item $\beta_x = \op{Ad}(\theta_x)\circ \alpha_x$ for all $x \in G$, 
\item $\omega_{x,y}\theta_{xy}=\theta_x\alpha_x(\theta_y)\sigma_{x,y}$.
\end{enumerate}
We often write $(\alpha, \sigma) \overset{\theta}{\sim} (\beta, \omega)$ when 
the twisted actions are equivalent via $\theta$.
\end{definition}

Fix two twisted dynamical systems $(G,A, \alpha, \sigma)$ and $(G,A, \beta, \omega)$ 
with $(\alpha, \sigma)$ and $(\beta, \omega)$ exterior equivalent with equivalence implemented by $\theta \colon G \to \op{Inv}_1(M(A))$. Let $\mathcal{R}$ be a nonempty family of contractive covariant representation of 
$(G, A, \alpha, \sigma)$. For each $(\pi, u) \in \mathcal{R}$ acting on a Banach space $E$ we define $v(\pi, u) \colon G \to \op{Iso}(E)$ by putting, for each $x \in G$, 
\[
v(\pi, u)_x \coloneqq v(\pi, u)x \coloneqq \widehat{\pi}(\theta_x)u_x.
\] 
We then let
\[
\mathcal{R}_\theta \coloneqq \{ (\pi, v(\pi, u)) \colon (\pi, u) \in \mathcal{R}\}.
\]
\begin{lemma}\label{cov+ISOM}
Let $(\alpha, \sigma) \overset{\theta}{\sim} (\beta, \omega)$.
Then,  for each $(\pi, u) \in \mathcal{R}$, the pair $(\pi, v(\pi, u)) \in \mathcal{R}_\theta$ is a contractive covariant representation of $(G,A, \beta, \omega)$. Furthermore, for each $f \in L^1(G,A,\alpha, \sigma)$, 
if $f^{(\theta)}(x)\coloneqq f(x)\theta_x^{-1}$ for each $x \in G$, then 
\[
\| f^{(\theta)} \|_{\mathcal{R}_\theta} = \| f \|_{\mathcal{R}}.
\]
\end{lemma}
\begin{proof}
For notational 
convenience, for the rest of the proof we write for short $v \coloneqq v(\pi,u)$ and we let $E$ 
be the Banach space on which both $\pi$ and $u$ are acting. The map $x \mapsto v_x\coloneqq  \widehat{\pi}(\theta_x)u_x$  is clearly continuous 
for each $x \in G$ and satisfies $\| v_x \xi\| \leq \| \xi \|$ for any $\xi \in E$. Next, for each $x \in G$, we define $t_x \in \Li(E)$ by $t_x\coloneqq u_x^{-1}\widehat{\pi}(\theta_x^{-1})$. Straightforward calculations show that $\| t_x \xi \| \leq \|\xi\|$ and that $t_xv_x=v_xt_x = \op{id}_E$, whence $v_x \in \op{Iso}(E)$ with $v_x^{-1}=t_x$. For any $x, y \in G$, below 
we compute using Condition \eqref{CovC3} from Definition \ref{CovariantRep} for $(\pi, u)$
in the second step, that $\theta$ implements the exterior equivalence 
in the third step, and Condition \eqref{CovC4} from Definition \ref{CovariantRep} for $(\pi, u)$
in the fifth step
\begin{align*}
v_x v_y   v_{xy}^{-1}& = \widehat{\pi}(\theta_x)u_x\widehat{\pi}(\theta_y)u_y t_{xy} \\
& =  \widehat{\pi}(\theta_x)u_x\widehat{\pi}(\theta_y)u_x^{-1}\widehat{\pi}(\sigma_{x,y} \theta_{xy}^{-1})\\
& = \widehat{\pi}(\theta_x)u_x\widehat{\pi}(\theta_y)u_x^{-1}\widehat{\pi}(\alpha_x(\theta_y^{-1})\theta_x^{-1}\omega_{x,y}) \\
& = \widehat{\pi}(\theta_x)u_x\widehat{\pi}(\theta_y)u_x^{-1}\widehat{\pi}(\alpha_x(\theta_y^{-1}))\widehat{\pi}(\theta_x^{-1})\widehat{\pi}(\omega_{x,y}) \\
& = \widehat{\pi}(\theta_x)u_x\widehat{\pi}(\theta_y)u_x^{-1}
u_x\widehat{\pi}(\theta_y^{-1})u_x^{-1}\widehat{\pi}(\theta_x^{-1})\widehat{\pi}(\omega_{x,y}) \\
& = \widehat{\pi}(\omega_{x,y}).
\end{align*}
This proves Condition \eqref{CovC3} for $(\pi, v)$. 
Next, for any $x \in G$ and $a \in A$ we use that $\theta$ implements the equivalence in the first step and 
Condition \eqref{CovC4} from Definition \ref{CovariantRep} for $(\pi, u)$ in the second step 
to obtain 
\[
\pi(\beta_x(a))  = \pi(\theta_x\alpha_x(a)\theta_x^{-1})  =  
\widehat{\pi}(\theta_x)u_x \pi(a) u_x^{-1}\widehat{\pi}(\theta_x^{-1})
= v_x \pi(a) v_x^{-1},
\]
proving Condition \eqref{CovC4} for $(\pi, v)$. Finally, observe that 
for each $v=v(\pi,u) \in \mathcal{R}_\theta$
\[
(\pi \rtimes v)(f^{(\theta)}) = \int_G \pi(f(x)\theta_{x}^{-1})\widehat{\pi}(\theta_x)u_x dx = \int_G \pi(f(x))u_x dx = (\pi \rtimes u)(f),
\]
whence 
\begin{align*}
\| f^{(\theta)} \|_{\mathcal{R_\theta}} &= 
\sup \{ \| (\pi \rtimes v)(f^{(\theta)}) \| \colon (\pi, v) \in \mathcal{R}_\theta\} \\
&= 
\sup \{ \| (\pi \rtimes u)(f) \| \colon (\pi, u) \in \mathcal{R}\} = \|f\|_{\mathcal{R}}
\end{align*}
as wanted. 
\end{proof}
The following result is a general version of \cite[Lemma 3.3]{PacRae89}. 
\begin{theorem}\label{Ext_Equiv_ISOM}  
Let $G$ be a locally compact group, let 
$A$ be a nondegenerate Banach algebra with a cai, and let $(\alpha, \sigma)$ and $(\beta, \omega)$ be two exterior equivalent twisted actions of $G$ on $A$ via $\theta$. Then the crossed product $F_{\mathcal{R}}(G, A, \alpha, \sigma)$ is isometrically isomorphic 
to $F_{\mathcal{R_\theta}}(G, A, \beta, \omega)$.
\end{theorem}
\begin{proof}
Let $B \coloneqq F_{\mathcal{R_\theta}}(G, A, \beta, \omega)$, define $k_A \colon A \to M(B)$ by $k_A \coloneqq (\lambda_A^{\mathcal{R_\theta}}, \rho_A^{\mathcal{R}_\theta})$, and define $k_G \colon G \to M(B)$
so that for each $x \in G$ we let
\[
k_G(x) \coloneqq \widehat{k_A}(\theta_x^{-1})(\lambda_G^{\mathcal{R_\theta}}, \rho_G^{\mathcal{R}_\theta})(x).
\]
 Analogous computations to the ones used in Lemma \ref{cov+ISOM}
show that $(k_A, k_G)$ is a covariant representation for $(G,A, \alpha, \sigma)$. 
Furthermore, a direct computation gives 
that for any $f \in L^1(G,A,\alpha, \sigma)$
\begin{equation}\label{k_aK_g}
(k_A \rtimes k_G)(f) =  ((\lambda_A^\mathcal{R_\theta}, \rho_A^\mathcal{R_\theta})  \rtimes (\lambda_G^\mathcal{R_\theta}, \rho_G^\mathcal{R_\theta}))(f^{(\theta)}), 
\end{equation}
where $f^{(\theta)}$ is as in the statement of Lemma \ref{cov+ISOM}. 
Thus, combining Equation \eqref{k_aK_g} with Proposition \ref{Universal_IntMap_Isom} and Lemma \ref{cov+ISOM} we get 
\[
\| (k_A \rtimes k_G)(f) \| = \| f^{(\theta)}\|_{\mathcal{R_\theta}} = \| f\|_{\mathcal{R}}.
\]
Furthermore, since for each $x \in G$, $\theta_x$ is invertible, 
it also follows from \eqref{k_aK_g} that
\[
(\iota_B \circ \tau_{R_\theta})(L^1(G,A,\beta, \omega)) \subseteq (k_A \rtimes k_G)(L^1(G,A,\alpha, \sigma)).
\]
Therefore $(k_A \rtimes k_G)(L^1(G,A,\alpha, \sigma))$ is  dense in $\iota_B(B)$. 
Hence  $B$ together with $(k_A, k_G)$ satisfy the universal conditions in Theorem \ref{Universal_R_Prop} 
proving that
$B$ is indeed isometrically isomorphic to $F_\mathcal{R}(G, A, \alpha, \sigma)$. 
\end{proof}

The full version (i.e., the Banach algebra generalization of \cite[Lemma 3.3]{PacRae89})
follows at once from the previous result. Indeed, if $\mathcal{C}=\mathcal{C}(G,A,\alpha, \sigma)$ consists of all the contractive representations of a twisted dynamical system $(G, A, \alpha, \sigma)$, 
we write $G \ltimes_{\alpha, \sigma} A \coloneqq F_{\mathcal{C}}(G,A,\alpha, \sigma)$. 
We get: 
\begin{corollary}
Let $G$ be a locally compact group, let 
$A$ be a nondegenerate Banach algebra with a cai, and let $(\alpha, \sigma) \overset{\theta}{\sim} (\beta, \omega)$ be two exterior equivalent twisted actions of $G$ on $A$. Then
$G \ltimes_{\alpha, \sigma} A$ is isometrically isomorphic to $G \ltimes_{\beta, \omega} A$.
\end{corollary}
\begin{proof}\label{FullC*-case}
It is clear that $\mathcal{C}(G,A,\alpha, \sigma)_\theta = \mathcal{C}(G,A,\beta, \omega)$, whence the desired result follows at once from the previous theorem. 
\end{proof}

The same proof yields a $p$-version 
for the full twisted crossed product of $L^p$-operator algebras:

\begin{corollary}\label{Cor:LpExtEquiv}
Let $p \in [1, \infty)$ and let $A$ be a nondegenerate $L^p$-operator algebra with a cai. If $(\alpha, \sigma)$ and $(\beta, \omega)$ are exterior equivalent twisted actions of a locally compact second countable group $G$ on $A$, then $F^p(G, A, \alpha, \sigma)$ is isometrically isomorphic to $F^p(G, A, \beta, \omega)$. 
\end{corollary}

Let $p \in (1, \infty)$. We are now ready for our final application: a $p$ version of the Packer-Raeburn untwisting trick \cite[Theorem 3.4]{PacRae89}. In what follows, fix a twisted dynamical system $(G, A, \alpha, \sigma)$ where

\begin{itemize}
\item $G$ is a second countable locally compact group, 
\item $A$ a nondegenerate separably representable $L^p$-operator algebra with a cai.
\end{itemize}
For notational convenience (see Notation \ref{notaRep}) we think of $A$ as a norm-closed subalgebra 
of $\Li(L^p(\mu))$ with $L^p(\mu)$ separable. We also fix the following notation for compact operators acting on $L^p(G)$: 
\[
\KG \coloneqq \mathcal{K}(L^p(G)) \subseteq \Li(L^p(G)).
\]
It is well known that $\KG$ is an $L^p$-operator algebra acting nondegenerately on $L^p(G)$ and that it has a cai. 
Recall also from Definition \ref{def_stabilization} that the left $p$-stabilization of $A$ in this case is 
\[
\op{St}_{\op{l}}^p(A)  = \op{St}_{\op{l}, G}^p(A)\coloneqq \KG \otimes_{\op{lsp}}^p A,
\]
which is an $L^p$-operator algebra whose norm is independent of the representation of  $A$ chosen.
\begin{lemma}\label{lem_caiST}
Let $p \in (1, \infty)$, let $A$ be a nondegenerate separably representable $L^p$-operator algebra with a cai. 
Then $\op{St}_{\op{l}}^p(A)$ also has a cai. 
\end{lemma}
\begin{proof}
Let $(u_\gamma)_{\gamma \in \Gamma}$ be a cai for $\KG$, 
let $(e_\lambda)_{\lambda \in \Lambda}$ be a cai for $A$. 
For  any $\pi_A \in \op{Rep}_p(A)$, a direct check shows that  
$((\iota_\mathcal{K} \otimes \pi_A) (u_\gamma \otimes e_\lambda))_{\gamma \in \Gamma, \lambda \in \Lambda}$
is a cai for $\KG \otimes_p \pi_A(A)$, so it follows that $(u_\gamma \otimes e_\lambda)_{\gamma \in \Gamma, \lambda \in \Lambda}$ is a cai for $\op{St}_{\op{l}}^p(A)$. 
\end{proof}
For each $x \in G$, we define 
\[
( \op{id}_{\KG} \otimes \alpha )_x \coloneqq (\op{id}_{\KG} \otimes \alpha)(x) \coloneqq \op{id}_{\KG} \otimes  \alpha_x.
\]
Then, $\op{id}_{\KG} \otimes  \alpha \colon G \to \op{Aut}(\op{St}_{\op{l}}^p(A))$. Similarly, if $1_{\KG} \in M(\KG)=\mathcal{B}(L^p(G))$ denotes the identity multiplier, for 
each $x, y \in G$ we define 
\[
(  1_{\KG} \otimes \sigma)_{x,y} \coloneqq (1_{\KG} \otimes \sigma)(x,y) \coloneqq 1_{\KG} \otimes \sigma_{x,y} \in \mathcal{B}(L^p(G)) \otimes_\op{lsp}^p  M(A) \subseteq M(\op{St}_{\op{l}}^p(A)).
\]
Hence,  $ 1_{\KG} \otimes \sigma \colon G \times G \to \op{Inv}_1(M(\op{St}_{\op{l}}^p(A)))$.

\begin{proposition}\label{Prop:Stable}
The pair $( \op{id}_{\KG} \otimes \alpha,1_{\KG} \otimes  \sigma )$ defined above
is a twisted action of $G$ on $\op{St}_{\op{l}}^p(A)$. Moreover, the twisted 
crossed product $F^p(G,\op{St}_{\op{l}}^p(A) ,  \op{id}_{\KG} \otimes \alpha,1_{\KG} \otimes  \sigma)$
is isometrically isomorphic to $\op{St}_{\op{l}}^p(F^p(G,A,\alpha,\sigma) )$. 
\end{proposition}
\begin{proof}
The first statement is a direct check that the conditions in Definition \ref{TwistedBanachSystem} are satisfied. For the second one, put $B\coloneqq \op{St}_{\op{l}}^p(F^p(G,A,\alpha,\sigma) )$ and let $\mathcal{R}^p = \mathcal{R}^p(G, A, \alpha, \sigma)$ be as in Definition \ref{defn_FullLpCP}. Theorem \ref{F_R:BAI} combined with Lemma \ref{lem_caiST} show that $B$ has a cai. The desired result now follows from the universal property (Theorem \ref{Universal_R_Prop}) with $k_A \coloneqq \op{id}_{\KG}  \otimes (\lambda_A^{\mathcal{R}^p}, \rho_A^{\mathcal{R}^p})$ and $k_G   \coloneqq 1_{\KG}  \otimes (\lambda_G^{\mathcal{R}^p}, \rho_G^{\mathcal{R}^p})$. 
\end{proof}

Let $p'$ be the Hölder conjugate of $p \in (1, \infty)$, that is $p' \coloneqq \frac{p}{p-1} \in (1, \infty)$. 
Denote by $(- \mid -) \colon L^{p'}(G) \times L^{p}(G) \to \C$ the usual duality pairing. For each $\varphi \in  L^{p}(G)$ and $\eta \in L^{p'}(G)$ we let $\Theta_{\varphi, \eta} \in \KG$ be the usual rank-1 map, where for any $\zeta \in L^p(G)$ and $x \in G$, we let
\begin{equation}\label{Theta_Map}
(\Theta_{\varphi, \eta} \zeta)(x) \coloneqq \varphi (x) (\eta \mid \zeta).
\end{equation}
Combining 
\cite[Example 4.5]{Ryan02} and \cite[Corollary 4.13]{Ryan02} we isometrically identify the Banach space injective tensor product 
$L^{p}(G) \otimes_{\varepsilon} L^{p'}(G)$ with $\KG$ via the map satisfying $\varphi \otimes \eta \mapsto \Theta_{\varphi, \eta}$. Furthermore, the first Proposition in \cite[Section 7.1]{defflor1993} states that the $\otimes_{p}$ norm dominates the injective one, which implies that 
the identity map on the algebraic tensor product $L^{p}(G) \otimes L^{p'}(G)$
extends to a contraction $L^{p}(G) \otimes_{p} L^{p'}(G) \to \KG$. 

We seek a generalization of the above contraction to the general $L^p$-operator algebra valued setting. We do so by considering the more general Banach space $L^{p}(G; L^{p'}(G;A)) = L^p(G) \otimes_p L^{p'}(G;A)$, 
whereas above we only dealt with the case $A=\C$. Observe that $L^{p}(G; L^{p'}(G,A))$ is readily
identified via 
\begin{equation}\label{eq_ident_xy}
\psi(x)y \leftrightarrow \psi(x,y)
\end{equation}
 with the space of 
measurable functions $\psi \colon G \times G \to A$ satisfying 
\begin{equation}\label{Eq:p,p'_norm}
\| \psi\|_{p,p'} \coloneqq \Bigg( \int_G \Bigg( \int_G \| \psi(x,y) \|_A^{p'} dy \Bigg)^{p/{p'}} dx \Bigg)^{1/{p}} <\infty.
\end{equation}
We remark that when $p=2$, under the identification in Equation \eqref{eq_ident_xy}, the space $L^2(G; L^{2}(G;A))$ coincides with $L^2(G \times G; A)$ and that $\|\psi\|_{2,2}=\| \psi\|_2$, so that the identification is isometric.

The following is a generalization, on the one hand, of \cite[Lemma 3.5]{PacRae89} (from $p=2$ to any $p\in (1,\infty)$), and on the other hand,
also a generalization of \cite[Exercise VI.3.7]{Con90}   (from $\C$-valued to $A$-valued). 

\begin{lemma}\label{K_psi_Lemma}
Let $p \in (1, \infty)$, let $p'=\frac{p}{p-1}$, let $\psi \in  L^{p}(G;L^{p'}(G;A))$, and let $\xi \in L^p(G; L^p(\mu)) = L^p(G) \otimes_p L^p(\mu) = L^p(\nu_G \times \mu)$. Then, the formula
\[
x \mapsto (K_\psi \xi )(x) \coloneqq \int_G \psi(x,y)\xi(y) dy
\]
defines a map $K_\psi \xi \colon G \to L^p(\mu)$ such that $K_\psi \in  \KG \otimes_p A \subseteq \Li(L^p(\nu_G \times \mu))$ and $\| K_\psi \| \leq \| \psi\|_{p,p'}$. Furthermore, the set 
$\{K_\psi \colon \psi \in L^{p}(G; L^{p'}(G;A))\}$ is dense in $\KG \otimes_p A$.
\end{lemma}

\begin{proof}
We need to be more careful than the case $A=\C$ discussed above as we are now tensoring with a general $L^p$-operator algebra $A$. To do so, consider a pure tensor $\psi_0 = \varphi \otimes \eta  \otimes a \in L^{p}(G) \otimes L^{p'}(G) \otimes A$, so that, under the identification in Equation \eqref{eq_ident_xy}, the function $\psi_0 \colon G \times G \to A$ is given by 
\[
\psi_0(x,y) = \varphi(x)\eta(y)a \ \text{ for all $x,y \in G$.  }
\]  
Then $\psi_0 \in  L^{p'}(G; L^{p}(G;A))$ with 
$\| \psi_0 \|_{p,p'} = \|\varphi\|_{p}\|\eta\|_{p'}\|a\|$. Further, if $\xi = \xi_1 \otimes \xi_2 \in L^p(G) \otimes_p L^p(\mu)$ and $\Theta$ is as in \eqref{Theta_Map}, a direct computation gives that for any $x \in G$ we have 
\[
(K_{\psi_0} \xi )(x) = \varphi(x) (\eta \mid \xi_1) a(\xi_2) = (\Theta_{\varphi, \eta}  \xi_1)(x) a(\xi_2) = [(\Theta_{\varphi, \eta} \otimes a)\xi](x).
\]
Hence, $K_{\psi_0}  = \Theta_{\varphi, \eta} \otimes a \in \KG \otimes_p A$ and $\| K_{\psi_0} \|=\|\Theta_{\varphi, \eta}\|\|a\|=\|\psi_0\|_{p,p'}$. More generally, recall first that finite sums of pure tensors, like $\psi_0$, are dense in $L^p(G;L^{p'}(G;A))$. 
Next, for a fixed $n \in \Z_{\geq 0}$, we define $\psi_n \in  L^p(G;L^{p'}(G;A))$ by
\[
\psi_n \coloneqq \sum_{j=1}^n \varphi_j \otimes \eta_j  \otimes a_j \in  L^{p}(G) \otimes L^{p'}(G) \otimes A.
\]
By linearity we have $K_{\psi_n} \in \KG \otimes_p A$. Then, for any $\xi \in L^p(\nu_G \times \mu)$, using Hölder's inequality for the spaces $L^{p'}(G;A)$ and $L^p(G; L^p(\mu))$ in the 
third step below, we obtain
\begin{align*}
\| K_{\psi_n}\xi \|_p^p & = \int_G \| (K_{\psi_n}\xi)(x) \|_p^p dx \\
& = \int_G \Bigg\| \int_G \Big( \sum_{j=1}^n \varphi_j(x)\eta_j(y)a_j\Big)\xi(y) dy \Bigg\|_p^p dx\\
& \leq \int_G \Bigg( \int_G \Big\| \sum_{j=1}^n \varphi_j(x)\eta_j(y)a_j\Big\|_{A}^{p'} dy\Bigg)^{p/p'} \|\xi\|_p^p  dx \\
& = \| \psi_n\|_{p,p'}^p \| \xi \|_p^p.
\end{align*}
From here it follows at once that $K_\psi \in \KG \otimes_p A$ and $\| K_\psi \| \leq \| \psi\|_{p,p'}$ for any $\psi \in L^{p}(G;L^{p'}(G;A))$. Finally, the density claim also follows, for we know that elements of the form $K_{\psi_n} = \sum_{j=1}^n \Theta_{\varphi_j, \eta_j} \otimes a_j$ are dense in $\KG \otimes_p A$ by construction.
\end{proof}

\begin{corollary}\label{cor_contrLSP}
The map $\psi \mapsto K_\psi$ defines a contraction from $L^{p}(G; L^{p'}(G;A))$ to 
$\op{St}_{\op{l}}^p(A)$ with dense range.
\end{corollary}
\begin{proof}
Take any $\pi_A \in \op{Rep}_p(A)$, say $\pi_A \colon A \to \mathcal{B}(L^p(\mu_0))$. Since $\pi_A$ is contractive, using the same notation as in the proof 
of Lemma \ref{K_psi_Lemma}, we have 
\begin{align*}
\Big\| \sum_{j=1}^n \varphi_j(x)\eta_j(y)\pi_A(a_j)\Big\|_{\mathcal{B}(L^p(\mu_0))} 
& = \Big\| \pi_A\Big( \sum_{j=1}^n \varphi_j(x)\eta_j(y)a_j \Big)\Big\|_{\mathcal{B}(L^p(\mu_0))} \\
& \leq \Big\| \sum_{j=1}^n \varphi_j(x)\eta_j(y)a_j\Big\|_{A}.
\end{align*}
Therefore, the same argument used in Lemma \ref{K_psi_Lemma} (replacing $L^p(\mu)$ by $L^p(\mu_0)$) 
shows that $\| (\iota_\mathcal{K} \otimes \pi_A)(K_\psi) \| \leq \| \psi\|_{p,p'}$ for any $\psi \in L^{p}(G;L^{p'}(G;A))$. 
From here we get at once that $\| K_\psi \|_{\op{lst}} \leq \| \psi\|_{p,p'}$ for any $\psi \in L^{p}(G;L^{p'}(G;A))$, as wanted. 
The density part is again clear from the tensor product construction.
\end{proof}

\begin{definition} \label{mult_op}
For any measurable function $f \colon G \to M(A)$, satisfying 
\[
\| f \|_{\infty} \coloneqq \sup \{ \| f(x) \| \colon {x \in G} \}<\infty,
\]
we define an operator $m(f) \colon L^p(G; L^p(\mu)) \to L^p(G;L^p(\mu))$ by 
letting, for each $\xi \in L^p(G; L^p(\mu))$ and $x \in G$, 
\[
(m(f)\xi)(x) \coloneqq f(x)\xi(x). 
\]
\end{definition}
\begin{remark}\label{mult_op_rmk}
It is clear that $m(f) \in \Li( L^p(G; L^p(\mu)))$ with $\| m(f) \| \leq \| f\|_\infty$. 
Furthermore, if $f(x) \in \op{Inv}_1(M(A))$ for all $x \in G$, then it follows that $m(f)$ is invertible 
with inverse $m(f^{-1})$ (where $f^{-1}(x) \coloneqq f(x)^{-1}$)  and $\| m(f) \| =1$. 
For $\psi \in L^p(G; L^{p'}(G;A))$ and all $x,y \in G$, 
we define  
\begin{equation}\label{Eq_action fromulas}
(f \cdot \psi)(x,y) \coloneqq f(x)\psi(x,y), \ \ (\psi \cdot f)(x,y) \coloneqq \psi(x,y) f(y).
\end{equation}
It follows at once from \eqref{Eq:p,p'_norm} that both $\| f \cdot \psi \|_{p,p'} $ and 
$\| \psi \cdot f \|_{p,p'} $ are bounded by $\| f \|_\infty \|\psi \|_{p,p'}$. 
Hence each measurable function $f \colon G \to M(A)$, with $\| f \|_\infty < \infty$, continuously acts on the left and right on 
$L^p(G; L^{p'}(G;A))$ via the formulas in Equation \eqref{Eq_action fromulas}. 
Now, if we let $K_\psi$ be as in the statement of Lemma \ref{K_psi_Lemma} above,  direct computations show that,  for any $\psi \in L^p(G; L^{p'}(G;A))$,  
\[
m(f)K_\psi = K_{f \cdot \psi}, \ \ K_\psi m(f) = K_{\psi \cdot f}.
\]
Therefore, in view of Equation \eqref{Ma=Dc} and Corollary
\ref{cor_contrLSP}, we conclude that 
$m(f) \in M(\op{St}_{\op{l}}^p(A))$ and that $m(f) \in \op{Inv}_1(M(\op{St}_{\op{l}}^p(A)))$ whenever $f(x) \in \op{Inv}_1(M(A))$ for all $x \in G$. 
\end{remark}

\begin{definition}\label{LRrep}
For each $\zeta \in L^p(G)$ and $x,y \in G$, let 
\[
(\lambda^{p}_x \zeta)(y) \coloneqq \zeta(x^{-1}y),
\]
The map $\lambda^{p} \colon G \to \Li(L^p(G))$ is the \textit{$p$-left regular representation of $G$}.
\end{definition}
\begin{remark}\label{Rmk_lambda:_multi}
Observe that for each $x \in G$, the operator $\lambda_x^p$ is an invertible isometry, whence $\| \lambda_x^p\|=1$, with $(\lambda_{x}^p)^{-1}= \lambda_{x^{-1}}^p$. Moreover, if $\Theta$ is as in \eqref{Theta_Map}, for any $\varphi \in L^p(G)$, $\eta \in L^{p'}(\mu)$, and any $x \in G$ we have 
\[
\lambda^p_{x} \Theta_{\varphi, \eta} = \Theta_{\lambda^p_{x}\varphi , \eta}, \ 
\Theta_{\varphi, \eta}\lambda^p_{x} =\Theta_{\varphi, (\lambda^{p'}_{x})^{-1}\eta }.
\]
Hence, in view of Equation \eqref{Ma=Dc} it follows that $\lambda^p_x \in \op{Inv}_1(M(\KG))$ for 
all $x \in G$. 
\end{remark}
We are now ready to state and prove
one of the main applications of the results developed so far: 
the $p$-version of the Packer-Raeburn untwisting trick. 

\begin{theorem}\label{PacRae-TRICK}
Let $p \in (1,\infty)$, let
$A$ be a nondegenerate separably representable $L^p$-operator algebra with a cai,
and let $(G, A, \alpha, \sigma)$ be a twisted dynamical system. 
Then there is an action $\beta$ of $G$ on $\op{St}_{\op{l}}^p(A)$
such that $\op{St}_{\op{l}}^p(F^p(G, A, \alpha, \sigma))$ 
is isometrically isomorphic to the untwisted crossed product 
$F^p(G, \op{St}_{\op{l}}^p(A) , \beta)$.  
\end{theorem}
\begin{proof}
Let $m$ be as in Definition \ref{mult_op} and define $\sigma_x^{-1} \colon G \to \op{Inv}_1(M(A))$ by letting for each $y \in G$, 
$\sigma_{x}^{-1}(y) \coloneqq \sigma_{x,y}^{-1}$. 
Define $\Sigma \colon G \to \op{Inv}_1(M(\op{St}_{\op{l}}^p(A)))$, such that if $x \in G$, 
\[
\Sigma_x \coloneqq \Sigma(x)\coloneqq m(\sigma_x^{-1}).  
\]
Next, let $\lambda^{p} \colon G \to \op{Inv}_1(M(\KG))$ be as in Definition \ref{LRrep} (see Remark \ref{Rmk_lambda:_multi}), and define $\theta \colon G \to \op{Inv}_1(M(\op{St}_{\op{l}}^p(A)))$
 by putting for each $x \in G$, 
\[
\theta_x \coloneqq  \theta(x) \coloneqq (\lambda_x^{p} \otimes \op{id}_A)\Sigma_x.
\]
We now show that $\Sigma$ and hence $\theta$ are strictly continuous. Take any sequence $(x_n)_{n=1}^{\infty}$ in $G$ 
with $x_n\to x \in G$. Using Remark \ref{mult_op_rmk}
and Corollary \ref{cor_contrLSP}
we obtain for any $\psi \in L^p(G;L^{p'}(G;A))$ 
\[
\| \Sigma_{x_n}K_\psi - \Sigma_{x}K_\psi \|_{\op{lsp}} \leq \| (\sigma_{x_n}^{-1}-\sigma_x^{-1})\cdot \psi \|_{p,p'}.
\]
This, together with the fact that 
$\sigma \colon G \times G \to \op{Inv}_1(M(A))$ is jointly strictly continuous and Dominated Convergence Theorem for $L^p(G;L^{p'}(G;A))$ show that
$\Sigma_{x_n}K_\psi \to \Sigma_xK_\psi$ in $\op{St}_{\op{l}}^p(A)$ for any $\psi \in L^p(G;L^{p'}(G;A))$. 
An analogous argument gives that 
$K_\psi\Sigma_{x_n} \to K_\psi\Sigma_x$ in $\op{St}_{\op{l}}^p(A)$ for any $\psi \in L^p(G;L^{p'}(G;A))$. 
Hence it follows from Corollary \ref{cor_contrLSP} that both $x \mapsto \Sigma_xK$ and  $x \mapsto K\Sigma_x$ are continuous maps 
$G \to  \op{St}_{\op{l}}^p(A)$ for any $K \in  \op{St}_{\op{l}}^p(A)$, and therefore $\Sigma$ is a strictly continuous map. This implies 
at once that the map $\theta$ is also strictly continuous. 

Finally, define $\beta \colon G \to \op{Aut}(\op{St}_{\op{l}}^p(A))$, so that for each $x \in G$, 
we put
\[
\beta_x \coloneqq \beta(x) \coloneqq \op{Ad}(\theta_x) \circ ( \op{id}_{\KG} \otimes \alpha )_x.
\]
 Since $\op{id}_{\KG} \otimes \alpha$ is strongly continuous and $\theta$ is strictly continuous, it follows that $\beta$ is strongly continuous. 
Then, the same algebraic manipulations done in the proof of 
\cite[Theorem 3.4]{PacRae89} show that $(\op{id}_{\KG} \otimes \alpha,  1_{\KG} \otimes  \sigma)  \overset{\theta}{\sim} (\beta, \mathbf{1})$, where $\mathbf{1}\colon G \times G \to \op{Inv}_1(M(\op{St}_{\op{l}}^p(A)))$ is the trivial twist $\mathbf{1}_{x,y} \coloneqq 1_{\op{St}_{\op{l}}^p(A)}$ for all $x,y \in G$.
The desired result now follows at once from combining Corollary \ref{Cor:LpExtEquiv} with Proposition \ref{Prop:Stable} and the fact that $F^p(G,\op{St}_{\op{l}}^p(A), \beta, \mathbf{1} ) = F^p(G,\op{St}_{\op{l}}^p(A), \beta)$ (see Remark \ref{Rmk_twist_1_ncp}). 
\end{proof}

We end the paper with an immediate consequence of the untwisting trick, which gives 
 that 
the K-theory of $L^p$-twisted crossed products can be computed 
using the already available tools for the untwisted case 
such as the ones developed in \cite[Section 6]{ncp2013CP}.

\begin{corollary}\label{cor: Ktheory}
Let $p \in (1,\infty)$, let
$A$ be a nondegenerate separably representable $L^p$-operator algebra with a cai,
and let $(G, A, \alpha, \sigma)$ be a twisted dynamical system. 
Then there is an action $\beta$ of $G$ on $\op{St}_{\op{l}}^p(A)$
such that 
\[
\op{K}_*( F^p(G, A, \alpha, \sigma) )\cong \op{K}_*(F^p(G, \op{St}_{\op{l}}^p(A) , \beta))
\]
\end{corollary}
\begin{proof}
$\op{St}_{\op{l}}^p(F^p(G, A, \alpha, \sigma))$ and $F^p(G, A, \alpha, \sigma)$ are Morita equivalent Banach algebras 
in the sense of Lafforgue (see \cite[Definition 1.1]{PAra15} ). The desired result now follows at once 
from combining Theorem \ref{PacRae-TRICK} with \cite[Theorem 1.2]{PAra15}.
\end{proof}

\bibliographystyle{plain}
\bibliography{LpTCP.bib}

\end{document}